\documentclass[11pt]{article}
\usepackage[T1]{fontenc}
\usepackage{lmodern}
\usepackage[margin=1in]{geometry}
\usepackage{amsmath,amssymb,amsthm,microtype}
\usepackage{needspace}
\usepackage{xcolor}
\usepackage[colorlinks=true,linkcolor=blue!45!black,citecolor=blue!45!black,urlcolor=blue!45!black,pdfencoding=auto]{hyperref}
\hypersetup{pdftitle={A Szemer\'edi--Trotter Theorem in Arbitrary Fields},pdfauthor={Mark Lewko}}
\newtheorem{theorem}{Theorem}
\newtheorem{lemma}[theorem]{Lemma}
\newtheorem{proposition}[theorem]{Proposition}
\newtheorem{corollary}[theorem]{Corollary}
\theoremstyle{remark}
\newtheorem{remark}[theorem]{Remark}
\newcommand{\F}{\mathbb F}
\newcommand{\LL}{\mathcal L}
\newcommand{\PP}{\mathcal P}
\newcommand{\Pol}{\Pi}
\allowdisplaybreaks[2]
\title{A Szemer\'edi--Trotter Theorem in Arbitrary Fields}
\author{Mark Lewko}
\date{}

\begin{document}
\maketitle

\begin{abstract}
Let $k$ be a field of characteristic $p\ge0$. We prove that $m$ points
and $n$ lines in $k^2$ determine at most $3(mn)^{2/3}+m+n+2mn/p$
incidences, the last term being omitted in characteristic zero. Over
the prime field $\F_p$ the coefficient of $mn/p$ can be replaced by $1$.
The proof uses the polynomial method, and for $m=n$ the bound is sharp
up to an absolute constant over prime fields.
As applications, over prime fields in which $-1$ is not a square we
obtain the $L^2\to L^r$ extension estimate for the paraboloid in
$\F_p^3$ for $r>10/3$. Over every odd prime field, we show that a
two-source extractor construction of Bourgain has exponentially small
error at every min-entropy rate greater than $1/3$. We also improve
sum-product estimates for small sets in positive characteristic and
obtain projection and Furstenberg estimates over prime fields.
The incidence inequalities with exact constants have been formalized in Lean.
\end{abstract}

\section{Introduction}

Let $k$ be a field, let $P\subset k^2$ be a finite set of points, and let
$\LL$ be a finite set of distinct lines in $k^2$, with coordinates $x,y$.
Write $m:=|P|$, $n:=|\LL|$, and
\[
 I(P,\LL):=|\{(q,\ell)\in P\times\LL:q\in\ell\}|.
\]
If $k$ has positive characteristic we denote it by $p$ and write $\F_p$
for its prime subfield. In characteristic zero we set $p=\infty$ and
interpret $1/p=0$. We write $A\lesssim B$, equivalently $B\gtrsim A$ or
$A=O(B)$, if $A\le CB$ for an absolute constant $C>0$. A subscript,
as in $A\lesssim_\varepsilon B$, indicates that $C$ may depend on the
subscripted parameter. We write $t_+:=\max\{t,0\}$ and use natural
logarithms unless a base is specified. For polynomials, $\deg F$ denotes
total degree, with subscripts specifying the variables counted. We write
$F_x:=\partial F/\partial x$ for the formal partial derivative, and
similarly for other variables.

Over the real numbers, the Szemer\'edi--Trotter theorem \cite{ST} states
that
\begin{equation}\label{eq:ST}
 I(P,\LL)\lesssim (mn)^{2/3}+m+n.
\end{equation}
T\'oth \cite{Toth} extended \eqref{eq:ST} to $\mathbb C$ (see also Zahl
\cite{Zahl}). Over any field of characteristic zero, the coordinates of
a finite configuration lie in a finitely generated subfield. Such a
field embeds in $\mathbb C$, so the estimate holds over every field of
characteristic zero.

In positive characteristic an additional term is necessary. The
configuration $P=\F_p^2$, $\LL=\{y=ax+b:a,b\in\F_p\}$ has $m=n=p^2$ and
$I(P,\LL)=p^3$. This configuration exists over every field of characteristic $p$.
Our main result is the following.

\Needspace{18\baselineskip}
\begin{theorem}\label{thm:main}
Let $k$ be a field, let $P\subset k^2$ be a finite set of points, and let
$\LL$ be a finite set of distinct affine lines in $k^2$. Write
$m=|P|$ and $n=|\LL|$. If $k$ has characteristic zero, then
\[
 I(P,\LL)\le 3(mn)^{2/3}+m+n.
\]
If $k$ has characteristic $p>0$, then
\begin{equation}\label{eq:main}
 I(P,\LL)\le 3(mn)^{2/3}+m+n+\frac{2mn}{p}.
\end{equation}
In positive characteristic, if $mn\le p^3$, then
\begin{equation}\label{eq:main-subcritical}
 I(P,\LL)\le 3(mn)^{2/3}+m+n.
\end{equation}
If $k=\F_p$ with $p$ prime, then for every $m,n$,
\begin{equation}\label{eq:main-prime}
 I(P,\LL)\le 3(mn)^{2/3}+m+n+\frac{mn}{p}.
\end{equation}
\end{theorem}

In particular, \eqref{eq:ST} holds in characteristic zero. In positive
characteristic it holds whenever $mn\le p^3$, and whenever
$\min(m,n)\le p$, since in the latter case
$mn/p\le\max(m,n)$. For $m=n=N$ the bound reads
$I(P,\LL)\lesssim N^{4/3}+N^2/p$, and the examples in
Section~\ref{sec:examples} show that this bound is sharp over prime fields for
every $1\le N\le p^2$. Over an extension field the last term depends on
the characteristic rather than on the field cardinality.

Bourgain, Katz, and Tao \cite{BKT} obtained the first nontrivial
incidence bounds over prime fields, via sum-product estimates. The previous
best incidence exponent for arbitrary point and line sets over fields of
characteristic $p$ is due to Stevens and de Zeeuw
\cite[Theorem~3]{SDZ}. In the balanced case their bound is
$O(N^{22/15})$ ($22/15\approx1.4667$) for $N\lesssim p^{15/11}$. For a finite field $\F_q$
of $q$ elements, Vinh \cite[Theorem~3]{Vinh} proved
$I(P,\LL)\le mn/q+\sqrt{qmn}$ by spectral methods.
The polynomial method has been applied to related problems by Dvir
\cite{Dvir}, in the proof of the finite-field Kakeya conjecture, and by
Guth and Katz \cite{GK}, in their work on distinct distances. Kaplan,
Matou\v{s}ek, and Sharir \cite[Section~3]{KMS} gave what is now the
standard polynomial proof of \eqref{eq:ST} over $\mathbb R$, based on
the polynomial partitioning theorem of Guth and Katz
\cite[Section~4]{GK}. An exposition is given in \cite{TaoPoly}.

Our proof uses the polynomial method together with an elementary
multiplicity estimate for a polynomial of two variables restricted to
parallel lines. Let $F$ be a nonzero polynomial of degree at most $d<p$
with no repeated irreducible factors. Such a polynomial is called
squarefree. Restrict $F$ to finitely many distinct parallel lines on
which it does not vanish identically. Parametrizing each line gives
a polynomial in one variable. Choose one point on each line
and consider the multiplicity of the restriction there, allowing
multiplicity zero when the polynomial is nonzero at the point.
Lemma~\ref{lem:roots} states that, after subtracting $1$ from each
positive multiplicity, their sum is at most $d(d-1)$.

To use the lemma we construct a polynomial in the plane variables $x,y$
and a third variable $z$, which stands for a slope, such that setting $z$
equal to a slope $a$ produces a polynomial vanishing on every given line
of slope $a$. The degree bound $d$ in $x,y$ is chosen later. Allowing
a suitable degree in $z$ gives enough coefficients to meet these
vanishing conditions without increasing $d$. At a point
lying on many given lines, these conditions, one for each slope, force
the polynomial to vanish to high order along an auxiliary line through
the point. Taking the auxiliary lines parallel, the multiplicity
estimate bounds the total excess multiplicity by $d(d-1)$, provided
the restrictions are nonzero. We count the points with identically
zero restrictions separately and combine the two estimates to bound
the number of incidences.
Section~\ref{sec:overview} gives a warm-up over $\mathbb C$ with the
same number of points and lines. It presents the construction and explains
how the degree choice gives the exponent $4/3$, assuming the multiplicity
estimate. We prove that estimate in Section~\ref{sec:roots} and give the
general polynomial argument in Section~\ref{sec:incidences}.
Appendix~\ref{app:constants}
contains the degree optimization for the stated constants and the finite-plane
estimate used to improve the characteristic term over $\F_p$.
Sections~\ref{sec:roots} and~\ref{sec:incidences}, together with that appendix,
give a self-contained proof of Theorem~\ref{thm:main}.

Section~\ref{sec:applications} gives several applications.
Over prime fields $\F_p$ in which $-1$ is not a square, we obtain the
$L^2\to L^r$ extension estimate for
the paraboloid in $\F_p^3$ for all $r>10/3$. This improves the range
$r>176/51\approx3.4510$ of \cite[Theorem~1.1]{LewBilII}, which combined the bilinear
approach of \cite{LewBilI} with incidence estimates.

It was observed in \cite{LewRect} that a sharp Szemer\'edi--Trotter
estimate would give a min-entropy rate near $3/8=0.375$ for Bourgain's
extractor based on the paraboloid in $\F_p^3$, with $-1$ not a square
in $\F_p$.
We obtain a better rate using his construction \cite{BourgainExt}
based on the parabola in $\F_p^2$. Over every odd prime field, the
extractor defined using $xy+x^2y^2$ works for two independent sources
of any min-entropy rate greater than $1/3$, with error exponentially
small in the input length. The proof iterates Theorem~\ref{thm:main} to bound higher
additive moments of arbitrary subsets of the parabola.\footnote{The rate $4/9\approx0.4444$ claimed for a higher-dimensional variant in \cite[Theorem~3]{LewExtractor} relied on an incorrect energy estimate in \cite[Lemma~4]{RS}. This issue does not impact the present work.}

For every $\varepsilon>0$ and every finite nonempty
$A\subset k$ with $|A|\le p^{41/71}$,
\[
 \max\{|A+A|,|AA|\}\gtrsim_\varepsilon |A|^{52/41-\varepsilon},
 \qquad 52/41\approx1.2683,
\]
where $A+A:=\{a+b:a,b\in A\}$ and $AA:=\{ab:a,b\in A\}$. There is no
size restriction in characteristic zero. For small sets in odd
characteristic the previous best exponent was $5/4$, up to logarithmic
factors, due to Mohammadi and Stevens \cite[Theorem~2]{MS} under the
hypothesis $|A|\lesssim p^{1/2}$.
A second argument uses Shakan's decomposition \cite{Shakan}, as
presented by Bloom \cite{Bloom}, with Theorem~\ref{thm:main} replacing
Szemer\'edi--Trotter. This gives the larger exponent
$49/38-\varepsilon$ ($49/38\approx1.2895$) in the smaller range $|A|\le p^{1/3}$.
We conclude with projection and Furstenberg estimates over prime
fields.

After the first version of this paper was posted to the arXiv,
Changxing Miao and Rui Xie informed me that they had independently
obtained a similar incidence estimate, along with several of the
applications given here \cite{MX}.

\section*{Acknowledgments and AI usage}

I would like to thank Josh Zahl for helpful suggestions on the
presentation and for many discussions of related questions over many
years, Paige Bright for pointing out the applications
to projections and Furstenberg sets in Section~\ref{sec:projections},
and Changxing Miao and Rui Xie for informing me of their work and
for related correspondence. I would also like to thank Christoph Thiele
for the wonderful summer schools he organizes. At the 2011 school,
he assigned me the paper of Solymosi and Tao \cite{SolTao}, which first exposed me to the subject.

ChatGPT~6 was used as a sounding board and for assistance proving
Lemma~\ref{lem:roots}. It and Claude Fable helped optimize the sum-product
and extractor estimates and revise the exposition. Claude
Fable suggested the approach in Section~\ref{sec:decomposition} and
helped develop that section. Claude and Codex were used to develop and
check the Lean formalization of the main incidence bounds and to revise
the explicit constants. The author takes responsibility for all content.

\section{A warm-up over the complex numbers}\label{sec:overview}

As a warm-up, let $P$ be a set of $N$ points and $\LL$ a set of $N$
lines in $\mathbb C^2$, with $N\ge1$. We will show that
\[
 I(P,\LL)\lesssim N^{4/3}.
\]
This case contains the polynomial construction and multiplicity count
used in the general proof. Over $\mathbb C$ we can choose an auxiliary
slope outside a finite exceptional set, and there is no characteristic
restriction on the degree. We use the multiplicity estimate stated below
and proved in Section~\ref{sec:roots}, without optimizing the constant here.

Fix an integer $d\ge1$. We will construct a nonzero squarefree polynomial
$F\in\mathbb C[x,y]$ of total degree at most $d$ and choose parallel
auxiliary lines, one through each point of $P$. These lines need not belong to
$\LL$. To make them distinct, choose their common slope to avoid the
finitely many slopes of nonvertical lines joining pairs of points of $P$.
Such a choice is possible since $\mathbb C$ is infinite.

The aim is to turn points on many given lines into roots of high
multiplicity on the auxiliary lines. Write $r(q)$ for the number of
nonvertical lines of $\LL$ through $q$. We will construct $F$ so that
its restriction to the auxiliary line through $q$ has multiplicity at
least $(r(q)-e)_+$ at $q$, unless the restriction is identically zero.
The nonnegative integer $e$ measures the loss in passing from incidences
to multiplicities, and the construction will give $e\lesssim N/d$.

For the nonzero restrictions, the multiplicity estimate bounds the sum
of the root multiplicities after subtracting $1$ from each positive
multiplicity. Allowing $e+1$ incidences at each point leaves a total of
at most $d(d-1)$ beyond these, giving $O(N(e+1)+d^2)$ incidences.

The identically zero restrictions require a separate count. Each such
auxiliary line gives a distinct linear factor of $F$, so there are at
most $d$ such points. Counting pairs of these points will bound their
nonvertical incidences by $O(N+d^2)$. The vertical lines contribute at
most $N$, since each point lies on at most one of them. Combining the
bounds and using $e\lesssim N/d$ gives
\[
 I(P,\LL)\lesssim N+\frac{N^2}{d}+d^2.
\]
The last two terms balance when $d$ is of order $N^{2/3}$, giving the
desired exponent $4/3$.

The multiplicity estimate we need is the complex case of
Lemma~\ref{lem:roots}. Let $F\in\mathbb C[x,y]$
be a nonzero squarefree polynomial of degree at most $d$. On each of
finitely many distinct parallel lines where $F$ does not vanish
identically, choose one point and let $\mu_\lambda$ be the order of
vanishing of the restriction there, with $\mu_\lambda=0$ if the
polynomial is nonzero at that point. Then
\begin{equation}\label{eq:overview-multiplicity}
 \sum_\lambda(\mu_\lambda-1)_+\le d(d-1).
\end{equation}
Thus a root of multiplicity one contributes zero to the sum, and a
root of multiplicity $\mu\ge2$ contributes $\mu-1$. The subtraction
of $1$ is necessary, since
$F=x$ has a simple root on every horizontal line.

\begin{proof}[Proof of the balanced bound]
Discard vertical lines, which contribute at most $N$ incidences,
and let $\LL_0$ be the remaining set, with $n=|\LL_0|\le N$.
We may assume $n\ge1$ and write the line equations as $y=ax+b$.

\smallskip
\Needspace{6\baselineskip}
\noindent\emph{Constructing the polynomial family.}
Fix an integer $d\ge1$, to be chosen at the end. We first leave the
slope as a variable $z$ and construct a nonzero polynomial
$S\in\mathbb C[x,y,z]$. Later, choosing a value $z_0$ for the common
auxiliary slope will give $F(x,y)=S(x,y,z_0)$. We seek $S$ of the form
\[
 S(x,y,z)=\sum_{i=0}^e z^iS_i(x,y),\qquad \deg S_i\le d,
\]
such that
\begin{equation}\label{eq:warmup-family}
 S(x,ax+b,a)\equiv0\qquad\text{for every }y=ax+b\text{ in }\LL_0.
\end{equation}
Here each $S_i$ lies in $\mathbb C[x,y]$, $d$ bounds the total degree
in $x,y$, and $e$ bounds the degree in $z$. For each slope $a$,
\eqref{eq:warmup-family} requires $S(x,y,a)$ to vanish identically
on every given line of slope $a$. Lines of different slopes impose
conditions on different members of this family. Increasing $e$
gives more coefficients to meet these conditions, while the degree
in $x,y$ remains at most $d$.

Each $S_i$ has $\binom{d+2}2$ coefficients, one for each monomial
$x^\alpha y^\beta$ with $\alpha+\beta\le d$, so there are
$(e+1)\binom{d+2}2$ unknown coefficients in all. On a given line
of slope $a$, the restriction $S(x,ax+b,a)$ has degree at most $d$ in
$x$. Requiring each of its coefficients to be zero imposes at most
$d+1$ homogeneous linear equations on the coefficients of $S$.
The $n$ given lines therefore impose at most $n(d+1)$ such equations.
There is a nonzero solution when the number of unknown coefficients
exceeds this bound, which is equivalent to
$(e+1)(d+2)>2n$. We may take
\[
 e=\left\lfloor\frac{2n}{d+2}\right\rfloor\le\frac{2n}{d+2}.
\]
We will need squarefreeness to apply \eqref{eq:overview-multiplicity}.
Replace $S$ by the product of its distinct irreducible factors.
Neither degree increases. If a product restricts to the zero
polynomial on a given line, at least one factor does, and that factor
remains in the new product. Thus the identities
\eqref{eq:warmup-family} are preserved and $S$ is now squarefree.

For comparison, the product of the equations of the $n$ given lines
satisfies \eqref{eq:warmup-family} with $e=0$, with no slope variable
at all. Its degree in the plane variables is $n$, and
\eqref{eq:overview-multiplicity} would then bound the excess
multiplicity only by $O(n^2)$, which can be as large as the trivial
bound $N^2$ on the number of incidences. The additional variable allows
us to decrease the degree in $x,y$, with $e\lesssim N/d$. We will see
that $e$ appears as a loss in the multiplicity at each point.

\smallskip
\noindent\emph{From incident lines to multiple roots.}
Fix $q=(q_1,q_2)\in P$, and recall that $r(q)$ counts the given
nonvertical lines through $q$. We now show how these lines force high
root multiplicity along the auxiliary line. For each fixed value of
$z\in\mathbb C$, the line of slope $z$ through $q$ is parametrized by
$(q_1+u,q_2+zu)$, $u\in\mathbb C$, with $u=0$ at $q$. Leaving $z$
as a variable, substitute this parametrization into $S$ and collect
powers of $u$:
\[
 R_q(u,z):=S(q_1+u,q_2+zu,z)=\sum_{j=0}^d c_{q,j}(z)u^j.
\]
For a fixed value $z=a$, the polynomial $R_q(u,a)$ is the restriction
of $S(x,y,a)$ to the line of slope $a$ through $q$, with $u=0$ at $q$.
It vanishes at $u=0$ to order at least $\mu$ exactly when
$c_{q,0}(a)=\dots=c_{q,\mu-1}(a)=0$. Thus the coefficients $c_{q,j}$,
which are polynomials in the single variable $z$, determine the order
of vanishing along the line of slope $z$ through $q$.

Now let $a$ be the slope of a given line through $q$, with intercept
$b$. Then $q_2=aq_1+b$, so the parametrized line is the given line
itself, and $R_q(u,a)$ is the restriction in \eqref{eq:warmup-family}
with $x=q_1+u$. Hence $R_q(u,a)\equiv0$ and every $c_{q,j}$ vanishes at $a$.
Distinct lines of $\LL_0$ through $q$ have distinct slopes, so each
coefficient $c_{q,j}$ has at least $r(q)$ distinct roots.

On the other hand, $\deg c_{q,j}\le e+j$. Indeed, a monomial
$x^\alpha y^\beta z^\gamma$ of $S$, with $\gamma\le e$, becomes
$(q_1+u)^\alpha(q_2+zu)^\beta z^\gamma$, and a term of degree $j$ in
$u$ uses at most $j$ of the factors $zu$, so its degree in $z$ is at
most $\gamma+j\le e+j$. In particular, $c_{q,0}$ has degree at most
$e$. If $r(q)>e$, then $c_{q,0}$ has more roots than its degree allows and must be
the zero polynomial. If $r(q)>e+1$, the same reasoning also forces
$c_{q,1}=0$. Continuing in this way, every coefficient of $u^j$ with
$j<r(q)-e$ is zero. Equivalently,
\begin{equation}\label{eq:warmup-contact}
 u^{(r(q)-e)_+}\mid R_q(u,z).
\end{equation}
The divisibility \eqref{eq:warmup-contact} holds for every value of
$z$, including slopes that do not occur in $\LL_0$. We will choose
the same value $z_0$ for all points of $P$. The restriction of $S(x,y,z_0)$ to
the line of slope $z_0$ through $q$ then has multiplicity at least
$(r(q)-e)_+$ at $q$, unless the restriction vanishes identically. These lines of
slope $z_0$ are auxiliary lines and need not belong to $\LL_0$.
The remaining step is to choose $z_0$ so that the auxiliary lines are
distinct and $S(x,y,z_0)$ is nonzero and squarefree, as required by
\eqref{eq:overview-multiplicity}.

\smallskip
\noindent\emph{Counting the incidences.}
Choose a complex number $z_0$ such that
\[
 F(x,y):=S(x,y,z_0)
\]
is nonzero and squarefree, and the lines
\[
 \lambda_q:\ y-q_2=z_0(x-q_1),\qquad q\in P,
\]
are distinct. Only finitely many values need to be excluded, so such
a $z_0\in\mathbb C$ exists. Indeed, since $S\ne0$, at least one of its
coefficients as a polynomial in $x,y$ is a nonzero polynomial in $z$.
Avoiding the roots of this coefficient ensures $F\ne0$. For
distinctness, avoid the slopes of the nonvertical lines joining pairs of points of
$P$. Substituting $z=z_0$ also preserves squarefreeness for all but
finitely many values of $z_0$.\footnote{If $S$ depends only on $z$ the assertion is
immediate. Otherwise let $S'$ be $S$ with its irreducible factors
depending only on $z$ removed. After a generic linear change of the
variables $x,y$, which does not affect squarefreeness, every
irreducible factor of $S'$ has positive degree in $x$ and the
coefficient of the highest power of $x$ in $S'$ is a nonzero polynomial
$L(z)$. Then $S'$ is squarefree and primitive as a polynomial in $x$
over $\mathbb C[y,z]$, so its discriminant with respect to $x$ is a
nonzero element of $\mathbb C[y,z]$. Expand the discriminant in powers
of $y$ and choose a nonzero coefficient, which is a polynomial in $z$.
Exclude the roots of the removed factors, of $L$, and of the chosen
coefficient. Then $S'(x,y,z_0)$ has the same degree in $x$ as $S'$, so
its discriminant is the specialization of the discriminant and is a
nonzero polynomial in $y$. Hence $S'(x,y,z_0)$ has no repeated factor
of positive degree in $x$, and it has no factor depending only on $y$,
since its leading coefficient in $x$ is a nonzero constant. Since $F$
is a nonzero constant multiple of $S'(x,y,z_0)$, it is squarefree.}

Let $E$ be the set of points $q\in P$ for which $F$ vanishes
identically on $\lambda_q$, and put $h=|E|$. We first count the
incidences outside $E$, where the restriction
$F(q_1+u,q_2+z_0u)=R_q(u,z_0)$ is nonzero. By
\eqref{eq:warmup-contact} its multiplicity at $u=0$ is at least
$(r(q)-e)_+$, so applying \eqref{eq:overview-multiplicity} to $F$ and
the lines $\lambda_q$ with $q\in P\setminus E$ gives
\[
 \sum_{q\in P\setminus E}(r(q)-e-1)_+\le d(d-1).
\]
Since $r(q)\le e+1+(r(q)-e-1)_+$, allowing $e+1$ incidences at each
point gives at most $N(e+1)+d(d-1)$ incidences outside $E$.

If $q\in E$ then $F$ vanishes identically on $\lambda_q$. Dividing
$F$ as a polynomial in $y$ by $y-q_2-z_0(x-q_1)$ leaves a remainder
depending only on $x$. Substituting $y=q_2+z_0(x-q_1)$ shows that this
remainder is zero, so the equation of $\lambda_q$ divides $F$. Distinct auxiliary
lines give linear factors that are not proportional, so the product of
these $h$ factors divides $F$, and $h\le\deg F\le d$. To count the
incidences between $E$ and the given lines, allow one incidence on each
given line. If a given line contains $s\ge1$ points of $E$, its
remaining $s-1$ incidences correspond to the pairs formed by fixing
one of these points and pairing it with the others. Each pair of points
lies on at most one given line. Thus
\[
 I(E,\LL_0)\le n+\binom h2\le n+\binom d2.
\]
Combining these estimates and adding back the vertical incidences gives
\[
 I(P,\LL)\le N(e+2)+d(d-1)+n+\binom d2
 \lesssim N+\frac{N^2}{d}+d^2.
\]
The last two terms balance at $d$ of order $N^{2/3}$. Taking
$d=\lfloor N^{2/3}\rfloor$ proves the bound.
\end{proof}

Over a finite field, the excluded slopes may exhaust the field.
In Section~\ref{sec:incidences} we therefore keep $z$ as a variable,
so we do not need to choose a slope in the original field. For $m$
points and $n$ lines, the term $N^2/d$ becomes $mn/d$, which balances
$d^2$ at $d=(mn)^{1/3}$. In characteristic $p$, the multiplicity
estimate requires $d<p$. When $(mn)^{1/3}\ge p$, we take $d=p-1$,
producing an additional term of order $mn/p$.

\section{The multiplicity lemma}\label{sec:roots}

We now prove the multiplicity estimate used in
Section~\ref{sec:overview}, over an arbitrary field. A squarefree
polynomial is restricted to distinct parallel lines, with one point
chosen on each line. The lemma bounds the sum of the root
multiplicities beyond the first order by the square of the degree.
In positive characteristic, the degree must be less than the
characteristic. The parallel lines let us use the same directional
derivative at every point.

Subtracting $1$ from each positive multiplicity allows us to count
vanishing conditions shared by the polynomial and its derivative.
Over a field $K$, if $f(u)=u^\mu g(u)$ with $g\in K[u]$ and $\mu\ge1$, then
$f'(u)=\mu u^{\mu-1}g(u)+u^\mu g'(u)$, so $u^{\mu-1}\mid f'$.
After removing the factors
constant in the chosen direction, the polynomial and its derivative
have no common nonconstant factor. A dimension count then bounds the
number of independent vanishing conditions they can both satisfy by the
product of their degrees. We give the count explicitly so that no
intersection-multiplicity theory is needed.

\begin{lemma}\label{lem:roots}
Let $K$ be a field, with $p$ its characteristic if positive and
$p=\infty$ otherwise, and let $d$ be a nonnegative integer with $d<p$.
Let $F\in K[x,y]$ be nonzero and squarefree of total degree at most $d$.
Let $\Lambda$ be a finite set of distinct parallel lines in $K^2$ on
none of which $F$ vanishes identically, and for each $\lambda\in\Lambda$
choose a point $q_\lambda\in\lambda$. Let $\mu_\lambda\ge0$ be the
multiplicity of the restriction of $F$ to $\lambda$ at $q_\lambda$: if
$v\in K^2$ is a nonzero vector parallel to $\lambda$ and $u$ is an
indeterminate, $\mu_\lambda$ is the largest integer $\mu$ with
$u^\mu\mid F(q_\lambda+uv)$ in $K[u]$. (Replacing $v$ by another such
vector rescales $u$ by a nonzero element of $K$, so $\mu_\lambda$ is
well defined.) In particular, $\mu_\lambda=0$ when $F(q_\lambda)\ne0$.
Then
\begin{equation}\label{eq:roots}
 \sum_{\lambda\in\Lambda}(\mu_\lambda-1)_+\le d(d-1).
\end{equation}
\end{lemma}

\begin{proof}
The conclusion is trivial if $\Lambda$ is empty. If the lines are
vertical, interchanging the coordinates and replacing $F(x,y)$ by
$F(y,x)$ reduces to the nonvertical case. Set $s:=|\Lambda|$, enumerate
the lines as $\lambda_1,\ldots,\lambda_s$, let $a\in K$ be their common
slope and $b_i$ the intercept of $\lambda_i$. The $b_i$ are distinct.
Write $q_i:=q_{\lambda_i}=(\xi_i,a\xi_i+b_i)$ and
$\mu_i:=\mu_{\lambda_i}$. The invertible linear change of coordinates
$X=x$, $Y=y-ax$ sends $\lambda_i$ to $Y=b_i$ and $q_i$ to $(\xi_i,b_i)$,
and transforms $F$ into
\[
 T(X,Y):=F(X,Y+aX),
\]
which is nonzero and squarefree of degree at most $d$, since the
substitution preserves degree and factorization into irreducibles.
Moreover $T(\xi_i+u,b_i)=F(q_i+u(1,a))$, so the restriction of $T$ to
$Y=b_i$ has multiplicity $\mu_i$ at $(\xi_i,b_i)$.

Factors of $T$ depending only on $Y$ are constant along each horizontal
line. They do not affect the multiplicity of a nonzero restriction, but
they are common factors of $T$ and $T_X$. We remove them. Write
$T=A(Y)G(X,Y)$, where $A$ is the product of the irreducible factors of
$T$ independent of $X$ (with $A=1$ if there are none) and $G:=T/A$. Then
$G$ is squarefree, and since $T(X,b_i)\ne0$ we have $A(b_i)\ne0$, so
$T(X,b_i)$ and $G(X,b_i)$ differ by a nonzero scalar and have the same
multiplicity at $X=\xi_i$. If $G$ is constant every $\mu_i$ is zero and
we are done.

Suppose $G$ is nonconstant. We claim that $G_X\ne0$ and that $G$ and
$G_X$ are coprime. Every irreducible factor $f$ of $G$ depends on $X$, and
$1\le\deg_Xf\le d<p$, so differentiating multiplies the leading
$X$-coefficient of $f$ by a nonzero element of $K$. Thus $f_X\ne0$ and
$\deg_Xf_X<\deg_Xf$. Write $G=fQ$. Squarefreeness gives $f\nmid Q$. If
$f\mid G_X$, then $G_X=f_XQ+fQ_X$ and unique factorization would give
$f\mid f_X$, which is impossible. This proves the claim, and in particular
$G_X\ne0$. The facts about unique factorization used here can be found in
\cite[Proposition~1.24 and Theorem~1.32]{Milne}.

Let $h_i:=(\mu_i-1)_+$. Both $G(\xi_i+u,b_i)$ and its $u$-derivative
$G_X(\xi_i+u,b_i)$ are divisible by $u^{h_i}$. This is trivial when
$h_i=0$. When $h_i>0$, it follows from the differentiation identity
preceding the lemma. Each divisibility requires the first $h_i$
coefficients of the restriction to be zero. We will show that these
conditions are independent for polynomials of sufficiently large degree.
Every multiple of $G$ or $G_X$ satisfies them, so comparing dimensions
bounds $\sum_ih_i$.

\Needspace{10\baselineskip}
Set $d_0:=\deg G$ and $d_1:=\deg G_X$, so $d_0\le d$ and $d_1\le d_0-1$.
For $j\ge0$ let $\Pol_j$ be the space of polynomials in $K[X,Y]$ of total
degree at most $j$, so $\dim\Pol_j=\binom{j+2}2$. Fix an integer
$M\ge d_0+d_1$, to be taken large, and let $W_M\subseteq\Pol_M$ be the
subspace of polynomials $H$ with
\[
 u^{h_i}\mid H(\xi_i+u,b_i)\qquad(1\le i\le s),
\]
that is, with the coefficients of $u^0,\ldots,u^{h_i-1}$ in
$H(\xi_i+u,b_i)$ equal to zero for each $i$. Each such coefficient is a
linear functional on $\Pol_M$, so $W_M$ is the common kernel of
$\sum_ih_i$ linear functionals.

These functionals are independent for $M$ large. To see this we
construct, for each prescribed coefficient, a polynomial for which that
coefficient is $1$ and all the others are $0$. We use Lagrange interpolation
in $Y$. For $1\le i\le s$ and
$0\le j<h_i$ put
\[
 B_{i,j}(X,Y):=(X-\xi_i)^j\prod_{\substack{1\le v\le s\\v\ne i}}
 \frac{Y-b_v}{b_i-b_v},
\]
the denominators being nonzero since the $b_i$ are distinct. Then
$B_{i,j}(\xi_i+u,b_i)=u^j$ and $B_{i,j}$ vanishes on every other line
$Y=b_v$, so exactly one of the prescribed coefficients of $B_{i,j}$
equals $1$ and the rest are $0$. For $M$ large all $B_{i,j}$ lie in
$\Pol_M$, and their linear combinations prescribe the coefficients
independently. Hence
\[
 \dim W_M=\dim\Pol_M-\sum_{i=1}^sh_i.
\]

On the other hand, every multiple of $G$ or of $G_X$ has the required
divisibility on each line, and so does every sum of such multiples. Thus
\[
 J_M:=G\Pol_{M-d_0}+G_X\Pol_{M-d_1}\subseteq W_M,
\]
where $G\Pol_{M-d_0}$ denotes the set of products $GU$ with
$U\in\Pol_{M-d_0}$. Since $G$ and $G_X$ are coprime, a polynomial
divisible by both is divisible by $GG_X$, so
\[
 G\Pol_{M-d_0}\cap G_X\Pol_{M-d_1}=GG_X\Pol_{M-d_0-d_1}.
\]
Multiplication by a nonzero polynomial is injective, so the dimension
formula for a sum of subspaces gives
\[
 \dim\Pol_M-\dim J_M
 =\binom{M+2}2-\binom{M-d_0+2}2-\binom{M-d_1+2}2+\binom{M-d_0-d_1+2}2
 =d_0d_1.
\]
Finally $J_M\subseteq W_M$ gives $\dim J_M\le\dim W_M$, so
\[
 \sum_{i=1}^sh_i=\dim\Pol_M-\dim W_M\le\dim\Pol_M-\dim J_M=d_0d_1\le d(d-1).
 \qedhere
\]
\end{proof}

\section{Proof of Theorem~\ref{thm:main}}\label{sec:incidences}

We construct the polynomial family and use the incidences to obtain
lower bounds for the multiplicities on distinct parallel auxiliary
lines. We then apply Lemma~\ref{lem:roots} and count the incidences at
points with identically zero restrictions separately. This gives an
estimate for each admissible degree. Appendix~\ref{app:constants} chooses
the degree and proves the prime-field refinement, completing the proof.

\begin{proof}[Proof of Theorem~\ref{thm:main}]
We may assume $m,n\ge1$. We first make every line nonvertical without
changing the number of points, lines or incidences. Introduce an
indeterminate $t$ and, over $k(t)$, apply the invertible linear map
$(x,y)\mapsto(tx+y,x)$. A line $y=ax+b$ becomes the graph
$y=(x-b)/(t+a)$, and a vertical line $x=c$ becomes $y=c$.
The denominators $t+a$ are nonzero. The field extension and this
change of coordinates preserve distinctness and incidences, and the
characteristic is unchanged. Replacing $k$ by $k(t)$ and relabeling the
configuration, we may therefore assume that every line of $\LL$ has a
unique equation $y=ax+b$ with $a,b\in k$.

\smallskip
\noindent\emph{Constructing the polynomial family.}
We say that $f\in k[x,y]$ vanishes identically on the line $y=ax+b$ if
$f(x,ax+b)$ is the zero polynomial in $k[x]$. This is a polynomial
identity. Over a finite field it is stronger than vanishing at the
points of the line in $k^2$. Over $\F_p$, for instance, $x^p-x$ vanishes
at every point of the $x$-axis without being divisible by $y$. We will
use polynomial identities throughout. The condition $f(x,ax+b)\equiv0$
is equivalent to $(y-ax-b)\mid f$. Indeed, dividing $f$ by the
polynomial $y-ax-b$, which is monic in $y$ over $k[x]$, gives
$f=(y-ax-b)Q+R$ with $Q\in k[x,y]$ and $R\in k[x]$, and substituting
$y=ax+b$ shows $R(x)=f(x,ax+b)$.

As in Section~\ref{sec:overview}, we construct a polynomial of small degree
in $x,y$ that also depends on the slope. Lines of different slopes will
correspond to different specializations. Introduce a further indeterminate $z$, which
will be specialized to the slope of a line, and seek a nonzero
$S\in k[x,y,z]$ satisfying
\begin{equation}\label{eq:interpolation}
 S(x,ax+b,a)\equiv0\quad\text{for every line }y=ax+b\text{ in }\LL,
\end{equation}
the identity being in $k[x]$ with $a,b$ fixed. Thus $S(x,y,a)$ vanishes
identically on every line of slope $a$ in $\LL$. Some specializations
$S(x,y,a)$ may be zero polynomials. Only $S$ itself is required to be
nonzero.

Fix an integer $d$ with $1\le d<p$, to be chosen in
Appendix~\ref{app:degree}. For each integer $e\ge0$ let $\mathcal V_{d,e}$ be the
$k$-vector space of polynomials
\[
 S(x,y,z)=\sum_{j=0}^ez^jS_j(x,y),\qquad S_j\in k[x,y],\quad\deg S_j\le d.
\]
Since there are $\binom{d+2}2$ monomials $x^\alpha y^\beta$ with
$\alpha+\beta\le d$,
\[
 \dim_k\mathcal V_{d,e}=(e+1)\binom{d+2}2.
\]
For a line $\ell:y=ax+b$ in $\LL$ and $S\in\mathcal V_{d,e}$, the
polynomial $S(x,ax+b,a)$ has degree at most $d$ in $x$. Let
$c_{\ell,i}(S)$ denote its coefficient of $x^i$, $0\le i\le d$. Each
$c_{\ell,i}$ is a linear function of the coefficients of $S$, and
\eqref{eq:interpolation} is equivalent to the system
\begin{equation}\label{eq:coefficient-equations}
 c_{\ell,i}(S)=0\qquad(\ell\in\LL,\ 0\le i\le d)
\end{equation}
of $n(d+1)$ homogeneous linear equations. A nonzero solution exists
whenever $(e+1)\binom{d+2}2>n(d+1)$, that is, whenever
$(d+2)(e+1)>2n$. Let
\begin{equation}\label{eq:e}
 e:=\left\lfloor\frac{2n}{d+2}\right\rfloor
\end{equation}
be the least nonnegative integer with this property, so that
$e\le2n/(d+2)$, and fix a nonzero $S\in\mathcal V_{d,e}$ satisfying
\eqref{eq:interpolation}. The two degree bounds have different roles.
The degree $d$ determines the bound in the multiplicity lemma, while
$e$ will be the loss in multiplicity at each point. The dimension count
allows us to decrease $e$ by increasing $d$. Without the slope variable, a nonzero
polynomial vanishing identically on every line of $\LL$ would be
divisible by the product of their $n$ distinct irreducible linear
equations and so would have degree at least $n$. Even when $n<p$,
Lemma~\ref{lem:roots} applied to it would bound the excess multiplicity
only by $O(n^2)$, which in the balanced case is no better than the
trivial bound. The slope variable is what permits $d$ to be small.

We may assume $S$ is squarefree. Indeed, replace $S$ by the product of its distinct irreducible
factors. The new polynomial divides the old one, so neither degree
increases. For each line $y=ax+b$ in $\LL$, the identity
$S(x,ax+b,a)\equiv0$ forces some irreducible factor of $S$ to have zero
restriction, since $k[x]$ is a domain, and that factor remains in the
new product.

\smallskip
\noindent\emph{Multiplicity at an incident point.}
The identities \eqref{eq:interpolation} involve one slope at a time. We
now combine the identities for all the lines through a given point.
We will use these identities to bound the order of vanishing along
the line of slope $z$ through that point.
For $q=(q_1,q_2)\in P$ let $r(q)$ be the number of lines of $\LL$
through $q$. Let $u$ be a further indeterminate and define
\[
 R_q(u,z):=S(q_1+u,q_2+zu,z)=\sum_{j=0}^dc_{q,j}(z)u^j\in k[u,z],
\]
so $c_{q,j}\in k[z]$ is the coefficient of $u^j$. For fixed $a\in k$,
$R_q(u,a)$ is the restriction of $S(x,y,a)$ to the line of slope $a$
through $q$, parametrized by $(x,y)=(q_1+u,q_2+au)$ with $u=0$ at $q$.
If $y=ax+b$ is a line of $\LL$ through $q$, then $q_2=aq_1+b$, so
$q_2+au=a(q_1+u)+b$ and \eqref{eq:interpolation} gives
$R_q(u,a)\equiv0$, that is, $c_{q,j}(a)=0$ for every $j$. Distinct lines
through $q$ have distinct slopes, so each $c_{q,j}$ vanishes at $r(q)$
distinct elements of $k$.

We compare these $r(q)$ roots with the bound $\deg c_{q,j}\le e+j$.
Indeed, a monomial $x^\alpha y^\beta z^\gamma$
of $S$, with $\alpha+\beta\le d$ and $\gamma\le e$, becomes
$(q_1+u)^\alpha(q_2+zu)^\beta z^\gamma$. A term of $u$-degree $j$
selects at most $j$ factors $zu$ from $(q_2+zu)^\beta$, and so has
$z$-degree at most $\gamma+j\le e+j$. Consequently $c_{q,j}$, having at
least $r(q)$ distinct roots, is the zero polynomial whenever
$j<r(q)-e$. This gives
\begin{equation}\label{eq:forced}
 u^{(r(q)-e)_+}\mid R_q(u,z).
\end{equation}

\smallskip
\noindent\emph{One multiplicity estimate for all the points.}
The divisibility \eqref{eq:forced} holds for each point separately. To
sum these multiplicities with Lemma~\ref{lem:roots}, we need a distinct
parallel line through every point. Specializing $z$ to an element
$a\in k$ can cause these lines to coincide, and
the specialization $S(x,y,a)$ need not be squarefree. If $k$ is finite,
there may be no slope that avoids these problems. We therefore keep
$z$ indeterminate and work over $K:=k(z)$, the field of rational
functions in $z$ with coefficients in $k$. We use $z$ as the common
slope of the auxiliary lines. The original configuration has the same
incidences over $K$, which has the same characteristic as $k$.
We now regard $S$ as an element of $K[x,y]$, of degree
at most $d$ in $x,y$, and $R_q$ as an element of $K[u]$.
For $q=(q_1,q_2)\in P\subset k^2\subset K^2$ let
\[
 \lambda_q:=\{(q_1+u,q_2+zu):u\in K\}\subset K^2,
\]
the auxiliary line through $q$. Then $R_q$ is the
restriction of $S$ to $\lambda_q$ with $q$ at $u=0$. These lines have
common slope $z$, and are distinct
for distinct $q,q'\in P$. Equal intercepts would give
$q_2-zq_1=q_2'-zq_1'$, and comparing coefficients of $z$ gives $q=q'$,
since $z$ is transcendental over $k$. Thus $\lambda_q\cap P=\{q\}$.
Even if several points lie on the same given line, their auxiliary
lines are distinct, as required by the lemma.

To apply the multiplicity lemma over $K$, we also need $S$ to remain
squarefree after this change of coefficient field. An irreducible
factor of $S$ involving $x$ or $y$ has no common
nonconstant divisor of its coefficients in $k[z]$, since such a divisor
would be a proper factor. By Gauss's lemma it remains irreducible
in $K[x,y]$. Two distinct such factors $f,g$ cannot
become scalar multiples over $K$. Indeed, clearing denominators would
give $af=bg$ with nonzero $a,b\in k[z]$. Since $f$ is prime in
$k[x,y,z]$ and cannot divide $b$, it would divide $g$, contradicting
their distinctness up to constant factors. Factors depending only on
$z$ become units of $K$, so $S$ remains squarefree in $K[x,y]$.

Lemma~\ref{lem:roots} requires the restriction to each line to be
nonzero, so set
\[
 E:=\{q\in P:R_q(u,z)\equiv0\}.
\]
For $q\in P\setminus E$ let $\mu(q)$ be the largest integer $\mu\ge0$
with $u^\mu\mid R_q$ in $K[u]$. This is the multiplicity of the
restriction of $S$ to $\lambda_q$ at $q$, and by \eqref{eq:forced}
$\mu(q)\ge(r(q)-e)_+$. Lemma~\ref{lem:roots}, applied to $S\in K[x,y]$
and the lines $\{\lambda_q:q\in P\setminus E\}$ with the point $q$ chosen
on $\lambda_q$, gives $\sum_{q\in P\setminus E}(\mu(q)-1)_+\le d(d-1)$,
and hence
\begin{equation}\label{eq:excess}
 \sum_{q\in P\setminus E}(r(q)-e-1)_+\le d(d-1).
\end{equation}
Since $r(q)\le e+1+(r(q)-e-1)_+$, the incidences outside $E$ number at
most $(m-|E|)(e+1)+d(d-1)$.

\smallskip
\noindent\emph{Identically zero restrictions.}
It remains to count the incidences on $E$. Each point of $E$ gives a
distinct factor of $S$ of degree one in $z$. This will bound $|E|$ by $e$.
If $q=(q_1,q_2)\in E$, substituting $u=x-q_1$ in $R_q\equiv0$ gives
$S(x,q_2+z(x-q_1),z)\equiv0$, and division by $y-q_2-z(x-q_1)$, which is
monic in $y$ over $k[x,z]$, gives $Q_q\in k[x,y,z]$ with
\[
 S(x,y,z)=\bigl(y-q_2-z(x-q_1)\bigr)Q_q(x,y,z).
\]
These factors are irreducible, being monic and linear in $y$, and
distinct for distinct $q$, since their coefficients determine $q$. Their
product therefore divides $S$. Each has degree one in $z$, so
$h:=|E|$ satisfies $h\le e$.

A line containing $s$ points of $E$ contributes $s\le1+\binom s2$
incidences, and two points lie on at most one common line, so
$I(E,\LL)\le n+\binom h2$. Since $h\le e$, we have
$\binom h2\le h(e+1)$. Thus the same allowance of $e+1$ incidences
per point covers the pair term on $E$. Combining with \eqref{eq:excess} yields
\begin{equation}\label{eq:parameter}
 I(P,\LL)\le(m-h)(e+1)+d(d-1)+n+\binom h2\le m(e+1)+d(d-1)+n.
\end{equation}
Using $e\le2n/(d+2)$, we obtain
\[
 I(P,\LL)\le m+n+\frac{2mn}{d+2}+d^2.
\]

This estimate holds for every integer $1\le d<p$. The terms
$2mn/(d+2)$ and $d^2$ balance at degree $(mn)^{1/3}$; in positive
characteristic the degree is capped at $p-1$.
Appendix~\ref{app:degree} checks this choice and obtains
the characteristic-zero bound, \eqref{eq:main}, and~\eqref{eq:main-subcritical}, including their explicit
constants. Appendix~\ref{app:prime} proves \eqref{eq:main-prime} for the
original configuration over $\F_p$ by combining the subcritical estimate
with a finite-plane second-moment count.
\end{proof}

The four incidence bounds of Theorem~\ref{thm:main} have been formalized
in Lean~\cite{LewLean}: the characteristic-zero bound,
\eqref{eq:main}, \eqref{eq:main-subcritical}, and \eqref{eq:main-prime}. The statements
count the cardinality of the set of incident pairs between a finite set
of points in $k^2$ and a finite set of one-dimensional affine subspaces
of $k^2$. They allow arbitrary fields in the first three cases; the last
is over $\mathbb Z/p\mathbb Z$ with $p$ prime. Thus the verified
characteristic term is $2mn/p$ over arbitrary fields of characteristic
$p$ and $mn/p$ over $\F_p$, and the linear term is $m+n$ in every case.
The proofs include the algebraic multiplicity estimate, the change of
coordinates, the degree choices, and the finite-plane second-moment count.
Their dependencies use only propositional extensionality, classical choice,
and quotient soundness, with no unproved proof obligations. The sharpness
examples in Section~\ref{sec:examples} and the applications in
Section~\ref{sec:applications} are not part of this formalization.
The formalization is registered in Palomar as
\href{https://palomar-registry.org/entry?id=PALOMAR-2026-10-07-000004\&version=1}{PALOMAR-2026-10-07-000004, version~1}.

\section{Examples}\label{sec:examples}

We first give a grid example showing that the main term is necessary. For an integer
$1\le j<p$ let $[j]:=\{1,\ldots,j\}\subset k$, integers being identified
with their images in $k$. For positive integers $r,s$ with $2rs<p$ put
\[
 P=[r]\times[2rs],\qquad\LL=\{y=ax+b:a\in[s],\ b\in[rs]\},
\]
so $m=2r^2s$ and $n=rs^2$. For $x\in[r]$ the integer representatives
satisfy $1\le ax+b\le2rs$, so each line contains exactly $r$ points of
$P$ and
\[
 I(P,\LL)=r^2s^2=2^{-2/3}(mn)^{2/3}.
\]
See also \cite[Example~5]{SDZ}. For $8\le N\le p^{3/2}$ take
$r=s=\lfloor N^{1/3}/2\rfloor$. Then $r\ge N^{1/3}/4$, $2r^3\le N$ and
$2r^2<p$, and adding points and lines until $m=n=N$ gives
$I(P,\LL)\ge N^{4/3}/256$. For $1\le N<8$ a single incidence suffices.
In characteristic zero there is no size restriction. One line with many
points, or one point on many lines, shows that the linear terms are
necessary.

To see why the term $mn/p$ is necessary, note that in $\F_p^2$ a random
line contains a given point with probability about $1/p$. Thus $m$
points and $n$ lines chosen at random give about $mn/p$ incidences. Fix an
$N$-point subset of $\F_p^2$ with $N\le p^2$, and choose $N$ of the
$p^2$ nonvertical lines uniformly at random. Exactly $p$ of the $p^2$
nonvertical lines pass through each point, and each line is selected
with probability $N/p^2$. The expected number of incidences is therefore
$N^2/p$, and some choice of lines
gives at least $N^2/p$ incidences.
Since $N^2/p\ge N^{4/3}$ for $N\ge p^{3/2}$, the balanced bound
$N^{4/3}+N^2/p$ is sharp throughout $1\le N\le p^2$. Over an extension
field $\F_q$ the dependence is on $p$ rather than $q$. The prime-subfield
example in the introduction rules out replacing $mn/p$ by $mn/q$. We do
not claim that the estimate is optimal at every scale over extension
fields.

The degree restriction in Lemma~\ref{lem:roots} is essential. In
characteristic $p$ consider $F(x,y):=x^p-y$, which is irreducible, hence
squarefree, being linear in $y$ with unit leading coefficient, and
satisfies $\partial F/\partial x=0$. For every $\xi\in k$ the restriction
of $F$ to the horizontal line $y=\xi^p$ is $(x-\xi)^p$, with
multiplicity $p$ at $(\xi,\xi^p)$. Since $\xi\mapsto\xi^p$ is injective,
an infinite field yields arbitrarily many distinct lines each
contributing $p-1$ to the left side of \eqref{eq:roots}, while
$\deg F=p$. Thus the lemma fails at degree $p$. In the incidence
proof, the restriction $d<p$ limits how small we can make the term
$mn/d$, producing the additional term $mn/p$.

\section{Applications}\label{sec:applications}

\subsection{Additive energy on the paraboloid}

In this subsection and the next, $p$ is a prime such that $-1$ is not
a square in $\F_p$, so that $v\cdot v=0$ for $v\in\F_p^2$ implies
$v=0$, where $\cdot$ is the usual dot product. This holds exactly when
$p\equiv3\pmod4$. For $A\subseteq\F_p^2$ put
\[
 \widetilde A:=\{(v,v\cdot v):v\in A\}\subseteq\F_p^3,
\]
and for a finite subset $B$ of an abelian group let
\[
 E(B):=|\{(b_1,b_2,b_3,b_4)\in B^4:b_1+b_2=b_3+b_4\}|
\]
denote its additive energy. We will bound this quantity for
$B=\widetilde A$ by counting right-angled corners in the plane,
following the approach of
Rudnev and Shkredov \cite[Section~4]{RS} and \cite{LewRect}. The bound
will be used in the next subsection for sets of all sizes.

\begin{proposition}\label{prop:energy}
If $A\subseteq\F_p^2$ and $N:=|A|\ge1$, then
\begin{equation}\label{eq:energy}
 E(\widetilde A)\lesssim N^2\log(2N)+\frac{N^3}p.
\end{equation}
\end{proposition}

\begin{proof}
An ordered quadruple counted by $E(\widetilde A)$ projects to $x,y,z,w\in A$ with
$w=x+z-y$ and
\[
 (x-y)\cdot(z-y)=0.
\]
The cases $x=y$ and $z=y$ contribute at most $2N^2$. Otherwise $x,y,z$
are distinct, since $x=z$ would give $(x-y)\cdot(x-y)=0$ and hence
$x=y$. The triple $(x,y,z)$ is then a corner, with perpendicular
supporting lines through $x,y$ and through $y,z$, and it determines $w$.
Discarding the condition $w\in A$ gives an upper bound, so it suffices
to count these corners.

For a line $\ell$ put $n_A(\ell):=|A\cap\ell|$, and for $t\ge1$ let
$\LL_t$ be the set of lines with $n_A(\ell)\ge t$. Theorem~\ref{thm:main}
gives
\[
 t|\LL_t|\lesssim(N|\LL_t|)^{2/3}+N+|\LL_t|+N|\LL_t|/p,
\]
so there is an absolute constant $C$ such that for
$t\ge t_0:=C(1+N/p)$ the last two terms can be absorbed into the left
side, yielding
\begin{equation}\label{eq:rich-app}
 |\LL_t|\lesssim N^2t^{-3}+Nt^{-1}\qquad(t\ge t_0).
\end{equation}

Associate each corner to the supporting line containing more points of
$A$, breaking ties arbitrarily. Corners associated to a line with fewer than
$t_0$ points contribute $O(N^2t_0)$. Once two adjacent vertices are
fixed, the third lies on a determined perpendicular line with fewer than
$t_0$ points of $A$. Now fix a line $\ell$ with $t\le n_A(\ell)<2t$. The
number of corners associated to $\ell$ is $O(t^3)$, since both supporting
lines have fewer than $2t$ points. The number is also $O(tN)$.
Since $v\cdot v\ne0$ for every nonzero $v$, the line $\ell$ is not
parallel to its perpendiculars. The perpendiculars through distinct points
of $\ell$ are therefore distinct parallel lines. Their intersections
with $A$ contain at most $N$ points in total, and for each such point
the remaining vertex on $\ell$ has fewer than $2t$ choices.
By \eqref{eq:rich-app} the contribution at scale $t$ is at most
\[
 \min\{t^3,tN\}\bigl(N^2t^{-3}+Nt^{-1}\bigr)\lesssim N^2,
\]
using $t^3$ for $t\le\sqrt N$ and $tN$ for $t\ge\sqrt N$. Summing over
$t=t_0,2t_0,4t_0,\ldots$ up to $N$ gives $O(N^2\log(2N))$, and together
with the contribution below $t_0$ this proves \eqref{eq:energy}.
\end{proof}

\subsection{The restriction problem in prime fields}

Let $\PP:=\{(v,v\cdot v):v\in\F_p^2\}$ be the paraboloid and let $d\sigma$
assign mass $p^{-2}$ to each of its points. For $s\in\F_p$ write
$e_p(s):=\exp(2\pi is/p)$, using any integer representative of $s$ in the
exponential. For $f:\PP\to\mathbb C$ the extension operator is
\[
 (f\,d\sigma)^\vee(x):=p^{-2}\sum_{\xi\in\PP}f(\xi)e_p(x\cdot\xi),
 \qquad x\in\F_p^3,
\]
with norms on $\F_p^3$ taken with respect to counting measure and norms
on $\PP$ with respect to $d\sigma$. The restriction problem asks how
small $r$ can be in the following inequality, with a constant
independent of $p$.

\begin{corollary}\label{cor:restriction}
For every $r>10/3$,
\[
 \|(f\,d\sigma)^\vee\|_{L^r(\F_p^3)}\lesssim_r\|f\|_{L^2(\PP,d\sigma)},
\]
uniformly over primes $p$ for which $-1$ is not a square in $\F_p$.
\end{corollary}

\begin{proof}
For $g:\F_p^3\to\mathbb C$ write
$\widehat g(\xi):=\sum_{x\in\F_p^3}g(x)e_p(-x\cdot\xi)$.
By duality, it suffices to prove
$\|\widehat g\|_{L^2(\PP,d\sigma)}\lesssim_s\|g\|_{L^s(\F_p^3)}$
for $1<s<10/7$. We first prove a bound in terms of the support size
when $|g|\le1$, then sum over dyadic level sets to treat general $g$.
We follow \cite[Lemma~10 and Section~5]{LewRect}. The energy estimate
in Proposition~\ref{prop:energy} handles intermediate support sizes.
For small and large supports we use Fourier decay and Plancherel,
respectively.

Let $g$ be nonzero with $|g|\le1$, let $G$ be its support, and put
$M:=|G|\ge1$. For $t\in\F_p$ let
$A_t:=\{v\in\F_p^2:(v,t)\in G\}$ be the slice of $G$ at height $t$,
and write $1_B$ for the indicator of $B$. Lemma~10 of \cite{LewRect}, together with the identity
$\|(1_{\widetilde A_t}d\sigma)^\vee\|_{L^4(\F_p^3)}
=p^{-5/4}E(\widetilde A_t)^{1/4}$, gives
\[
 \|\widehat g\|_{L^2(\PP,d\sigma)}
 \lesssim M^{1/2}+p^{-1/8}M^{3/8}
 \Bigl(\sum_{t\in\F_p}E(\widetilde A_t)^{1/4}\Bigr)^{1/2}.
\]
Empty slices have zero energy. Since $\sum_t|A_t|=M$, H\"older gives
$\sum_t|A_t|^{1/2}\le p^{1/2}M^{1/2}$ and
$\sum_t|A_t|^{3/4}\le p^{1/4}M^{3/4}$.
Using Proposition~\ref{prop:energy} and $|A_t|\le p^2$ therefore yields
\begin{equation}\label{eq:restriction-middle}
 \|\widehat g\|_{L^2(\PP,d\sigma)}
 \lesssim M^{1/2}+\bigl(\log(2p)\bigr)^{1/8}p^{1/8}M^{5/8}+p^{-1/8}M^{3/4}.
\end{equation}
For $p^{5/3}\le M\le p^{5/2}$ the right side is
$O((\log(2p))^{1/8}M^{7/10})$. For $M\le p^{5/3}$ the elementary
Fourier-decay bound
$\|\widehat g\|_{L^2(\PP,d\sigma)}\lesssim M^{1/2}+p^{-1/2}M$ gives
$O(M^{7/10})$, and for $M\ge p^{5/2}$ Plancherel gives
$\|\widehat g\|_{L^2(\PP,d\sigma)}\le p^{1/2}M^{1/2}\le M^{7/10}$.
See \cite[Section~5]{LewRect} for these two estimates. The logarithm
is needed only in the middle range, where $M\ge p^{5/3}$ and hence
$\log(2p)\le\log(2M)$. We have therefore proved
\begin{equation}\label{eq:restricted-support}
 \|\widehat g\|_{L^2(\PP,d\sigma)}
 \lesssim M^{7/10}\bigl(\log(2M)\bigr)^{1/8}
 \qquad(|g|\le1).
\end{equation}

Fix $1<s<10/7$ and now let $g$ be any function with
$\|g\|_{L^s(\F_p^3)}=1$. Then $|g|\le1$. For $j\ge0$ put
\[
 G_j:=\{x:2^{-j-1}<|g(x)|\le2^{-j}\},\qquad M_j:=|G_j|,
\]
so that $M_j\le2^{s(j+1)}$. Applying \eqref{eq:restricted-support} to
$2^jg1_{G_j}$ and summing by the triangle inequality gives
\begin{align*}
 \|\widehat g\|_{L^2(\PP,d\sigma)}
 &\lesssim\sum_{j:M_j>0}2^{-j}M_j^{7/10}
       \bigl(\log(2M_j)\bigr)^{1/8}\\
 &\lesssim_s\sum_{j\ge0}(j+1)^{1/8}2^{-j(1-7s/10)}
 \lesssim_s1.
\end{align*}
The last series converges because $s<10/7$. Duality gives the stated
range, since the conjugate exponent of $10/7$ is $10/3$.
\end{proof}

\subsection{Two-source extraction}\label{sec:extractor}

We now apply Theorem~\ref{thm:main} to Bourgain's construction
\cite{BourgainExt} to obtain a two-source extractor with min-entropy
rate near $1/3$ and exponentially small error. The aim is to produce
a bit taking each value with probability close to $1/2$, assuming
only that the two sources are independent and have sufficiently large
min-entropy. Following the terminology of
\cite{LewExtractor}, a rate near $1/3$ means every fixed rate greater
than $1/3$. Throughout this subsection $p$ is any odd prime. The
min-entropy of a finite random variable is
$H_\infty(X)=-\log_2\max_x\Pr(X=x)$.
Thus $H_\infty(X)\ge\rho\log_2p$ means that no single value has
probability greater than $p^{-\rho}$.

Let $\beta_p:\F_p\to\{0,1\}$ be the indicator of
$\{0,\ldots,(p-1)/2\}$, where we identify $\F_p$ with the residues
$\{0,\ldots,p-1\}$. We consider Bourgain's construction
\[
 \operatorname{Ext}_p(x,y)=\beta_p(xy+x^2y^2),
 \qquad x,y\in\F_p.
\]
The polynomial $xy+x^2y^2$ is the dot product of $(x,x^2)$ and
$(y,y^2)$.

\begin{corollary}\label{cor:extractor}
Fix $\rho>1/3$. There is $c_\rho>0$ such that if $X,Y$ are independent
$\F_p$-valued random variables with
$H_\infty(X),H_\infty(Y)\ge\rho\log_2p$, then
\[
 \left|\Pr\bigl(\operatorname{Ext}_p(X,Y)=1\bigr)-\frac12\right|
 \lesssim_\rho p^{-c_\rho}.
\]
Thus the error is exponentially small in the input length $\log_2p$.
\end{corollary}

We follow the Fourier analytic approach of \cite{BourgainExt,LewExtractor}.
To bound the output bias, we will prove cancellation in exponential
sums of $xy+x^2y^2$. We obtain this cancellation from bounds for higher
additive moments, which count pairs of tuples on the parabola with the same sum.
For $A\subseteq\F_p$ and an integer $j\ge2$, let $J_j(A)$ count the
solutions in $A^{2j}$ to
\[
 a_1+\cdots+a_j=a'_1+\cdots+a'_j,\qquad
 a_1^2+\cdots+a_j^2=(a'_1)^2+\cdots+(a'_j)^2.
\]

\begin{lemma}\label{lem:parabola-moments}
If $N=|A|\ge1$ and $j\ge2$, then
\begin{equation}\label{eq:parabola-moments}
 J_j(A)\lesssim_j N^{2j-3+2^{2-j}}+\frac{N^{2j-1}}p.
\end{equation}
\end{lemma}

\begin{proof}
All convolutions and norms in this proof use counting measure on
$\F_p^2$. Put $\Gamma=\{(a,a^2):a\in A\}$ and
$r_j=1_\Gamma^{*j}$, so $r_j(v)$ counts ordered $j$-tuples of points
of $\Gamma$ with sum $v$. Then
\[
 \|r_j\|_1=N^j,\qquad \|r_j\|_2^2=J_j(A),\qquad
 \|r_j\|_\infty\le2N^{j-2}\quad(j\ge2).
\]
For the last bound, fix $j-2$ entries. The sum and sum of squares of
the remaining two determine their product, since $p$ is odd, and
there are at most two ordered pairs with a given sum and product.
In particular $J_2(A)\le2N^2$.

The functions $r_j$ count tuples with multiplicity, so we need a
weighted form of Theorem~\ref{thm:main}. For
finitely supported nonnegative weights $f$ on points and $g$ on lines,
let $I(f,g)$ be the sum of $f(q)g(\ell)$ over incident pairs. Then
\begin{equation}\label{eq:weighted-ST}
\begin{split}
 I(f,g)\lesssim{}&
 (\|f\|_1\|g\|_1)^{1/3}(\|f\|_2\|g\|_2)^{2/3}\\
 &+\|f\|_1\|g\|_\infty+\|g\|_1\|f\|_\infty
 +p^{-1}\|f\|_1\|g\|_1.
\end{split}
\end{equation}
Indeed, express each weight as the integral of the indicators of its
superlevel sets and apply Theorem~\ref{thm:main}. The main term follows
from
\[
 \int_0^\infty|\{f>t\}|^{2/3}\,dt
 \lesssim\|f\|_1^{1/3}\|f\|_2^{2/3}.
\]
For nonzero $f$, this inequality follows by integrating
$\min\{\|f\|_1/t,\|f\|_2^2/t^2\}^{2/3}$, split at
$t=\|f\|_2^2/\|f\|_1$. The other terms integrate to those displayed
in \eqref{eq:weighted-ST}.

For $j\ge3$, the convolution identity gives
\[
 J_j(A)=\sum_{u,v\in\F_p^2}
 r_j(u)r_{j-1}(v)1_\Gamma(u-v).
\]
Subtracting $u_1^2$ from the second coordinate turns translates of the
parabola into lines. If $u-v\in\Gamma$, the point $(u_1,u_2-u_1^2)$ lies on the line
\[
 y=-2v_1x+v_2+v_1^2.
\]
Both transformations are injective, and dropping the extra requirement
$u_1-v_1\in A$ gives an upper bound. Applying
\eqref{eq:weighted-ST} with weights $r_j$ on points and $r_{j-1}$ on lines,
and writing $J_j=J_j(A)$,
therefore gives
\[
 J_j\lesssim N^{(2j-1)/3}(J_jJ_{j-1})^{1/3}
              +N^{2j-3}+N^{2j-1}/p.
\]
If the last two terms do not already bound $J_j$ up to an absolute
constant, we can absorb them into the left side and divide by
$J_j^{1/3}$. In either case,
\begin{equation}\label{eq:moment-recurrence}
 J_j\lesssim N^{j-1/2}J_{j-1}^{1/2}
              +N^{2j-3}+N^{2j-1}/p.
\end{equation}
Starting with $J_2\le2N^2$, induction proves
\eqref{eq:parabola-moments}. The first term in the bound for $J_{j-1}$
gives $N^{2j-3+2^{2-j}}$. The term involving $p$ gives
$N^{2j-2}/\sqrt p$, which is at most
$\tfrac12N^{2j-3}+\tfrac12N^{2j-1}/p$.
\end{proof}

We next use the moment bounds to estimate exponential sums. We state
the following inequality in $\F_p^d$ for use in
Remark~\ref{rem:paraboloid-extractor}. If
$d\ge1$, $\mu,\nu$ are probability measures on $\F_p^d$, and $j\ge1$
is an integer, then
\begin{equation}\label{eq:moment-amplification}
 \left|\sum_{v,w}\mu(v)\nu(w)e_p(v\cdot w)\right|^{j^2}
 \le p^{d/2}\|\mu^{*j}\|_2\|\nu^{*j}\|_2.
\end{equation}
Here $\mu^{*j}$ is the distribution of the sum of $j$ independent
random variables, each with distribution $\mu$, and similarly for $\nu$.
To prove this, write $T$ for the sum on the left and put
$F(v)=\sum_w\nu(w)e_p(v\cdot w)$. Jensen's inequality gives
$|T|^j\le\sum_v\mu(v)|F(v)|^j$. Choose complex numbers
$\theta(v)$ of modulus one such that
$|F(v)|^j=\theta(v)F(v)^j$, choosing $\theta(v)=1$ when $F(v)=0$.
Set $G(u)=\sum_v\mu(v)\theta(v)e_p(v\cdot u)$. Expanding gives
\[
 |T|^j\le\left|\sum_u\nu^{*j}(u)G(u)\right|
 \le\sum_u\nu^{*j}(u)|G(u)|.
\]
Since $\nu^{*j}$ is a probability measure, a second application of
Jensen gives
$|T|^{j^2}\le\sum_u\nu^{*j}(u)|G(u)|^j$.
Choose complex numbers $\zeta(u)$ of modulus one with
$|G(u)|^j=\zeta(u)G(u)^j$. Expanding $G(u)^j$ and applying
Cauchy--Schwarz and Plancherel gives
\begin{align*}
 |T|^{j^2}
 &\le\left|\sum_{u,v}\zeta(u)\nu^{*j}(u)(\mu\theta)^{*j}(v)e_p(v\cdot u)\right|\\
 &\le p^{d/2}\|\nu^{*j}\|_2\|(\mu\theta)^{*j}\|_2
 \le p^{d/2}\|\nu^{*j}\|_2\|\mu^{*j}\|_2,
\end{align*}
which proves \eqref{eq:moment-amplification}.

\begin{proof}[Proof of Corollary~\ref{cor:extractor}]
Choose $1/3<\tau<\min\{\rho,1/2\}$ and put
$N=\lfloor p^\tau\rfloor$. A probability distribution with all point
masses at most $1/N$ is a convex combination of uniform distributions
on $N$-element sets. To see this, if two entries lie strictly between
$0$ and $1/N$, move mass between them in both directions until one of
the entries equals $0$ or $1/N$. The original vector is a convex
combination of the two resulting vectors. A probability vector cannot have exactly one such entry,
so iteration ends at vectors with $N$ entries equal to $1/N$.
The entropy hypotheses and independence therefore reduce the proof to
uniform sources on arbitrary $N$-element sets $A,B\subseteq\F_p$. The
output probability is bilinear in the two source distributions, so
the absolute bias is bounded by the average of the absolute biases
for the uniform sources.

Put $\gamma_j=3-2^{2-j}$. Since $\gamma_j\to3$ and $\tau>1/3$, we
may choose a fixed $j\ge2$ such that $\tau\gamma_j>1$.
Since $N\le\sqrt p$,
Lemma~\ref{lem:parabola-moments} gives
$J_j(A),J_j(B)\lesssim_j N^{2j-\gamma_j}$.
Apply \eqref{eq:moment-amplification} with $d=2$ to the uniform
probability measures on $\{\lambda(a,a^2):a\in A\}$ and
$\{(b,b^2):b\in B\}$, for any $\lambda\ne0$. Multiplying every point
by $\lambda$ preserves the convolution norms, so
\begin{equation}\label{eq:scalar-character}
 \left|\frac1{N^2}\sum_{a\in A,\,b\in B}
 e_p\bigl(\lambda(ab+a^2b^2)\bigr)\right|^{j^2}
 \le p\frac{\sqrt{J_j(A)J_j(B)}}{N^{2j}}
 \lesssim_j pN^{-\gamma_j}.
\end{equation}
Thus the normalized exponential sum in \eqref{eq:scalar-character} is
$O_\rho(p^{-\eta})$ for every $\lambda\ne0$, where
$\eta=(\tau\gamma_j-1)/j^2>0$.

Write the normalized Fourier expansion as
\[
 \beta_p(t)=\sum_{\lambda\in\F_p}
 b_\lambda e_p(\lambda t),\qquad
 b_\lambda=p^{-1}\sum_{t\in\F_p}\beta_p(t)e_p(-\lambda t).
\]
Then $b_0=1/2+1/(2p)$. For $\lambda\ne0$, summing a geometric
progression gives $|b_\lambda|\lesssim1/|\lambda|_p$, where
$|\lambda|_p$ is the absolute value of the representative in
$[-(p-1)/2,(p-1)/2]$. Thus
$\sum_{\lambda\ne0}|b_\lambda|\lesssim\log(2p)$.
Combining this with \eqref{eq:scalar-character} bounds the output bias
by $O_\rho(p^{-\eta}\log(2p)+p^{-1})$. The corollary follows by taking
$0<c_\rho<\eta$.
\end{proof}

\begin{remark}\label{rem:paraboloid-extractor}
Assume in this remark that $-1$ is not a square in $\F_p$. The same
argument applies to the paraboloid variant of the construction
considered in \cite{LewExtractor,LewRect}, in which
\[
 \operatorname{Ext}^{(2)}_p(x,y):=\beta_p\bigl(x\cdot y+(x\cdot x)(y\cdot y)\bigr),
 \qquad x,y\in\F_p^2,
\]
so the quantity inside $\beta_p$ is the dot product of $(x,x\cdot x)$ and $(y,y\cdot y)$ in
$\F_p^3$ and the inputs have length $2\log_2p$. It was observed in
\cite{LewRect} that a bound $E(\widetilde A)\lesssim_\varepsilon
N^{2+\varepsilon}$ for small $A\subseteq\F_p^2$ would give this
extractor min-entropy rate near $3/8=0.375$. Proposition~\ref{prop:energy}
gives this bound, and we include the proof. Let $\rho>3/8$ and let
$X,Y$ be independent $\F_p^2$-valued random variables with
$H_\infty(X),H_\infty(Y)\ge2\rho\log_2p$. Choose $3/8<\tau<\min\{\rho,1/2\}$
and put $N=\lfloor p^{2\tau}\rfloor\le p$. As in the proof of
Corollary~\ref{cor:extractor}, it suffices to treat uniform sources on
$N$-element sets $A,B\subseteq\F_p^2$. Apply
\eqref{eq:moment-amplification} with $d=3$ and $j=2$ to the uniform
measures on $\lambda\widetilde A$ and $\widetilde B$, where
$\lambda\ne0$. The uniform measure $\mu$ on an $N$-element set $U$
has $\|\mu^{*2}\|_2^2=N^{-4}E(U)$, and dilation preserves energy, so
Proposition~\ref{prop:energy} and $N\le p$ give
\[
 \left|\frac1{N^2}\sum_{a\in A,\,b\in B}
 e_p\bigl(\lambda(a\cdot b+(a\cdot a)(b\cdot b))\bigr)\right|^4
 \le p^{3/2}\frac{\sqrt{E(\widetilde A)E(\widetilde B)}}{N^4}
 \lesssim\frac{p^{3/2}\log(2N)}{N^2}
 \lesssim p^{3/2-4\tau}\log(2p).
\]
Since $\tau>3/8$, the Fourier expansion of $\beta_p$ gives exponentially
small error, as before. The rate near $1/3$ in
Corollary~\ref{cor:extractor} comes from iterating the moment estimate in
Lemma~\ref{lem:parabola-moments}. This uses the fact that
$u-v\in\Gamma$ can be written as a point-line incidence in $\F_p^2$.
For the paraboloid the corresponding condition gives point-plane
incidences in $\F_p^3$, so Theorem~\ref{thm:main} does not give the
same recurrence.
\end{remark}

\subsection{Sum-product estimates}

Here $k$ is again an arbitrary field, with $p$ as in the introduction.
We give two sum-product estimates. The argument below gives exponent
$52/41-\varepsilon$ for $|A|\le p^{41/71}$ ($41/71\approx0.5775$). The decomposition argument
in the next subsection gives the larger exponent $49/38-\varepsilon$
in the smaller range $|A|\le p^{1/3}$. In characteristic zero there is
no size restriction, so the second estimate is stronger in that case.
Both arguments use a theorem of Bloom \cite{Bloom} which bounds $|A+A|$
from estimates for the number of representations of a sum $a+b$, with $a\in A$ and $b$
in an auxiliary set $B$.

\begin{corollary}\label{cor:sumproduct}
For every $\varepsilon>0$ and every finite nonempty $A\subset k$ with
$|A|\le p^{41/71}$,
\[
 \max\{|A+A|,|AA|\}\gtrsim_\varepsilon|A|^{52/41-\varepsilon}.
\]
In characteristic zero there is no restriction on $|A|$.
\end{corollary}

\begin{proof}
The assertion is trivial for a singleton, and deleting $0$ from a larger
set loses at most half its elements, so we may assume $0\notin A$. Put
\[
 N:=|A|,\qquad\Delta:=\frac{\max\{|A+A|,|AA|\}}N\ge1.
\]
For a finite nonempty $B\subset k$ write
$A+B:=\{a+b:a\in A,\ b\in B\}$ and let
$r_{A+B}(s):=|\{(a,b)\in A\times B:a+b=s\}|$ count representations of
$s$ as a sum from $A$ and $B$. Let $\kappa(A)$ be the least number such that
\begin{equation}\label{eq:control-definition}
 \sum_{s\in k}r_{A+B}(s)^3\le\kappa(A)N^2|B|^2
 \quad\text{for every finite nonempty }B\subset k.
\end{equation}
We call $\kappa(A)$ the additive control of $A$.
The bounds $r_{A+B}\le\min(N,|B|)$ and $\sum_sr_{A+B}(s)=N|B|$ give
$\kappa(A)\le1$, and $B=\{0\}$ gives $\kappa(A)\ge1/N$. Bloom's theorem
\cite[Theorem~1]{Bloom}, which applies in any abelian group and in
particular in the additive group of $k$, states that for every
$\delta>0$
\begin{equation}\label{eq:bloom}
 |A+A|\gtrsim_\delta N\kappa(A)^{-11/19+\delta}.
\end{equation}
As Bloom notes, this sumset bound is essentially due to Rudnev and
Stevens \cite{RudStev}. We will show that a small product set forces
$\kappa(A)$ to be small. The bound \eqref{eq:bloom} then forces
$A+A$ to be large. More precisely, we will prove
\begin{equation}\label{eq:control-bound}
 \kappa(A)\lesssim\log(2N)\Bigl(\frac{\Delta^2}N+\frac{\Delta N}p\Bigr).
\end{equation}
To prove \eqref{eq:control-bound}, fix $B$, put $M:=|B|$, and let
$X\subset k$ be finite. Apply
Theorem~\ref{thm:main} to the points $A^{-1}\times X$, where
$A^{-1}:=\{a^{-1}:a\in A\}$, and the distinct lines
$y=ax+b$ with $a\in AA$, $b\in B$. For $s\in X$ and $b\in B$ with
$s-b\in A$, each $c\in A$ gives the incidence
$(c^{-1},s)\in\{y=c(s-b)x+b\}$, and the point and line together determine
$c,s,b$, so these incidences are distinct. Since $|AA|\le\Delta N$,
dividing by $N$ gives
\begin{equation}\label{eq:sum-level}
 \sum_{s\in X}r_{A+B}(s)
 \lesssim\Delta^{2/3}N^{1/3}(|X|M)^{2/3}+|X|+\Delta M+\frac{\Delta NM}p|X|.
\end{equation}
Let $X_t:=\{s:r_{A+B}(s)\ge t\}$. For $t\ge t_0:=C(1+\Delta NM/p)$, with
$C$ a sufficiently large absolute constant, the terms proportional to
$|X_t|$ can be absorbed into the left side, giving
\begin{equation}\label{eq:sum-tail}
 |X_t|\lesssim\Delta^2NM^2t^{-3}+\Delta Mt^{-1}.
\end{equation}

If $p<\infty$ and $M>p/\Delta$, the bounds $r_{A+B}\le N$ and
$\sum_sr_{A+B}(s)=NM$ give
\[
 \frac{\sum_sr_{A+B}(s)^3}{N^2M^2}\le\frac NM<\frac{\Delta N}p.
\]
We may therefore assume $M\le p/\Delta$, with $p=\infty$ in characteristic zero.
Terms with $r_{A+B}<t_0$ contribute at most $t_0^2NM$ to the third
moment. A dyadic range $t\le r_{A+B}<2t$ contributes
$O(t^3|X_t|)$. Summing \eqref{eq:sum-tail} over dyadic $t$ from $t_0$
up to $\min(N,M)$ therefore bounds the remaining terms by
\[
 O\bigl(\Delta^2NM^2\log(2N)+\Delta M\min(N,M)^2\bigr).
\]
After division by $N^2M^2$ the contribution from $r_{A+B}<t_0$ is
\[
 O\Bigl(\frac1{NM}+\frac{\Delta^2NM}{p^2}\Bigr)\lesssim\frac1N+\frac{\Delta N}p,
\]
using $\Delta M\le p$, and the contribution from $r_{A+B}\ge t_0$ is
$O(\Delta^2N^{-1}\log(2N)+\Delta/N)$, using $\min(N,M)^2\le NM$. Since
$\Delta\ge1$, this proves \eqref{eq:control-bound} uniformly in $B$.

Choose $0<\delta<11/19$ and put $\alpha:=11/19-\delta$. Combining
\eqref{eq:bloom} and \eqref{eq:control-bound} according to which term of
\eqref{eq:control-bound} is larger gives
\[
 \Delta\gtrsim_\delta\min\Bigl\{
 \Bigl(\frac N{\log(2N)}\Bigr)^{\alpha/(1+2\alpha)},
 \Bigl(\frac p{N\log(2N)}\Bigr)^{\alpha/(1+\alpha)}\Bigr\},
\]
the second term being omitted in characteristic zero. As $\delta\to0$
the exponents tend to $11/41$ and $11/30$ respectively, and the size
hypothesis gives $p/N\ge N^{30/41}$. Choosing $\delta$ small in terms of
$\varepsilon$ and absorbing the logarithms yields
$\Delta\gtrsim_\varepsilon N^{11/41-\varepsilon}$, as required.
\end{proof}

\subsection{Shakan's decomposition and a second sum-product bound}\label{sec:decomposition}

Shakan \cite[Theorem~1.10]{Shakan} showed that every finite
$A\subset\mathbb R$ contains two large subsets such that the additive
control of one and the multiplicative control of the other have
product $O(|A|^{-1})$, up to logarithms. The subsets may overlap.
Bloom \cite[Theorems~11 and~13]{Bloom} gave a proof using the
Szemer\'edi--Trotter bound for translates of a convex graph, and
combined the decomposition with \eqref{eq:bloom} to obtain exponent
$49/38-\varepsilon$ \cite[Theorems~3 and~14]{Bloom}. In the
multiplicative setting the corresponding incidences are between points
and lines, so we can use Theorem~\ref{thm:main}. The additional $mn/p$
term leads to the restriction $|A|^3\le p$.

For a finite nonempty $U\subset k$ let $\kappa^+(U)$ be the additive
control defined by \eqref{eq:control-definition} with $U$ in place of
$A$. Write $k^*:=k\setminus\{0\}$. For finite nonempty
$U,B\subset k^*$ let $r_{UB}(s):=|\{(u,b)\in U\times B:ub=s\}|$, and
let $\kappa^\times(U)$ be the least number such that
$\sum_sr_{UB}(s)^3\le\kappa^\times(U)|U|^2|B|^2$ for every finite
nonempty $B\subset k^*$. Applying \eqref{eq:bloom} in the multiplicative
group $k^*$ shows that small multiplicative control forces a large
product set.

We record three facts about control in an
abelian group. Write $\kappa$ for either $\kappa^+$ or $\kappa^\times$.
We use additive notation for both groups.
First, $1/|U|\le\kappa(U)\le1$.
Second, if $U\subseteq U'$ then $\kappa(U)|U|^2\le\kappa(U')|U'|^2$,
since $r_{U+B}\le r_{U'+B}$ pointwise.
Third, if $U_1,\ldots,U_t$ are disjoint then
$\kappa(U_1\cup\cdots\cup U_t)\le\sum_i\kappa(U_i)$ \cite[Lemma~3]{Bloom}.
Indeed, for every $B$ the triangle inequality in $\ell^3$ and H\"older's
inequality give
\[
 \|r_{(\bigcup_iU_i)+B}\|_3\le\sum_i\|r_{U_i+B}\|_3
 \le\sum_i\kappa(U_i)^{1/3}|U_i|^{2/3}|B|^{2/3}
 \le\Bigl(\sum_i\kappa(U_i)\Bigr)^{1/3}\Bigl(\sum_i|U_i|\Bigr)^{2/3}|B|^{2/3}.
\]
We also use the trivial bound
$\sum_sr_{U+B}(s)^3\le|U|^3|B|$, valid because $r_{U+B}\le|U|$ and
$\sum_sr_{U+B}(s)=|U||B|$. Consequently, to prove $\kappa(U)\le\kappa_0$,
it suffices to consider sets $B$ with $|B|<|U|/\kappa_0$.

\begin{proposition}\label{prop:decomposition}
Let $A\subset k^*$ be finite with $N:=|A|\ge1$ and $N^3\le p$. Then
there are $X,Y\subseteq A$ with $|X|,|Y|\ge N/2$ and
\[
 \kappa^+(Y)\,\kappa^\times(X)\lesssim\frac{(\log4N)^5}N.
\]
\end{proposition}

\begin{proof}
The main step is the following claim. Let $T\subset k^*$ be finite and
nonempty with $|T|^3\le p$, and put $\kappa:=\kappa^+(T)$ and
$L:=\log(4|T|)$. Then there is $T'\subseteq T$ with
$|T'|\gtrsim\kappa|T|/L^3$ and
\begin{equation}\label{eq:claim}
 \kappa\,\kappa^\times(T')\lesssim L^5\frac{|T'|}{|T|^2}.
\end{equation}
We prove the claim in four steps. We first choose a set of sums with
comparable numbers of representations, then select the elements of
$T$ that occur frequently in these representations. We use incidences
to bound the multiplicative control of the resulting set.

\smallskip
\noindent\emph{Step 1.}
By the definition of $\kappa$, there is a finite nonempty $B_0\subset k$ with
$\sum_sr(s)^3\ge\frac12\kappa|T|^2|B_0|^2$, where $r:=r_{T+B_0}$.
The bounds $\sum_sr(s)^3\le|B_0|^2\sum_sr(s)=|T||B_0|^3$ and
$\sum_sr(s)^3\le|T|^3|B_0|$ give
\begin{equation}\label{eq:B0-size}
 \kappa|T|\le2|B_0|,\qquad|B_0|\le2|T|/\kappa.
\end{equation}
The elements with $r(s)<\kappa|T|/4$ contribute at most
$(\kappa|T|/4)^2\sum_sr(s)=\kappa^2|T|^3|B_0|/16\le\kappa|T|^2|B_0|^2/8$
to $\sum_sr(s)^3$. The remaining elements satisfy
$\kappa|T|/4\le r(s)\le|T|$ and contribute at least
$\frac38\kappa|T|^2|B_0|^2$. Since $\kappa\ge1/|T|$ they lie in $O(L)$
dyadic ranges $\delta|T|\le r(s)<2\delta|T|$ with $\delta=2^{-i}$, so
there is $\delta\ge\kappa/8$ for which
$S:=\{s:\delta|T|\le r(s)<2\delta|T|\}$ satisfies
$8\delta^3|T|^3|S|\gtrsim\kappa|T|^2|B_0|^2/L$, that is,
\begin{equation}\label{eq:S-size}
 |S|\gtrsim\frac{\kappa|B_0|^2}{L\delta^3|T|}.
\end{equation}
Every $s\in S$ has at least $\delta|T|$ representations $s=t+b$ with
distinct $b\in B_0$, so $|B_0|\ge\delta|T|$.

\smallskip
\noindent\emph{Step 2.}
For $x\in k$ let $\rho(x):=|\{(s,b)\in S\times B_0:s-b=x\}|$, so
$\rho(x)\le|S|$. Then
$\sum_{x\in T}\rho(x)=\sum_{s\in S}r(s)\ge\delta|T||S|$. The elements
$x\in T$ with $\rho(x)<\delta|S|/4$ contribute at most
$\delta|T||S|/4$, and the rest lie in $O(L)$ dyadic ranges, since
$\delta\ge1/(8|T|)$. Thus there is $\delta/8\le\eta\le1$ for which
$T':=\{x\in T:\eta|S|\le\rho(x)<2\eta|S|\}$ satisfies
\begin{equation}\label{eq:Tprime-size}
 |T'|\gtrsim\frac{\delta|T|}{L\eta}.
\end{equation}

\smallskip
\noindent\emph{Step 3.}
To obtain the required lower bound for $|T'|$, we bound $\eta$ in
terms of $\delta$ and $\kappa$.
The sum $\sum_{x\in k}\rho(x)^2$ counts quadruples
$(s,b,s',b')\in S\times B_0\times S\times B_0$ with $s-b=s'-b'$.
Since $s,b,b'$ determine $s'$, the number of such quadruples is at most
$|S||B_0|^2$. Hence
$\eta^2|S|^2|T'|\le|S||B_0|^2$. Inserting \eqref{eq:Tprime-size} and
then \eqref{eq:S-size} gives
\begin{equation}\label{eq:eta-kappa}
 \eta\kappa\lesssim L^2\delta^2.
\end{equation}
In particular $|T'|\gtrsim\delta|T|/(L\eta)\gtrsim\kappa|T|/(L^3\delta)
\ge\kappa|T|/L^3$, as claimed.

\smallskip
\noindent\emph{Step 4.}
We now estimate $\kappa^\times(T')$ using Theorem~\ref{thm:main}.
Let $B\subset k^*$ be finite and nonempty and write
$r_\times:=r_{T'B}$. For $x\in T'$, $b\in B$, and a representation
$x=s-b_0$ with $(s,b_0)\in S\times B_0$, the point $(s,xb)$ lies on
the line $\ell_{b,b_0}:y=b(u-b_0)$, in coordinates $(u,y)$, since
$b(s-b_0)=bx$. The lines $\ell_{b,b_0}$ are distinct, and the point
and the line together determine $s,b,b_0$ and hence $x$. Thus for every
finite $C\subset k$, each pair $(x,b)$ with $xb\in C$ contributes at
least $\eta|S|$ distinct incidences between the points $S\times C$ and
the lines $\ell_{b,b_0}$, and Theorem~\ref{thm:main} gives
\[
 \eta|S|\sum_{c\in C}r_\times(c)
 \lesssim(|S||C||B||B_0|)^{2/3}+|S||C|+|B||B_0|+\frac{|S||C||B||B_0|}p.
\]
For $i\ge0$ let $C_i:=\{c:2^i\le r_\times(c)<2^{i+1}\}$. At most
$O(L)$ of these sets are nonempty, since $r_\times\le|T|$. We have
$\sum_cr_\times(c)^3\le8\sum_i2^{2i}\sum_{c\in C_i}r_\times(c)$, and
by H\"older's inequality
$\sum_i2^{2i}|C_i|^{2/3}\lesssim L^{1/3}(\sum_cr_\times(c)^3)^{2/3}$,
while $\sum_i2^{2i}|C_i|\le\sum_cr_\times(c)^2\le|T'||B|^2$ and
$\sum_i2^{2i}\le2\max r_\times^2\le2|T'||B|$. Applying the incidence
bound with $C=C_i$ and summing over $i$,
\[
 \sum_cr_\times(c)^3
 \lesssim\frac{L^{1/3}(|S||B||B_0|)^{2/3}}{\eta|S|}
 \Bigl(\sum_cr_\times(c)^3\Bigr)^{2/3}
 +\frac{|T'||B|^2}{\eta|S|}
 \Bigl(|S|+|B_0|+\frac{|S||B_0||B|}p\Bigr).
\]
If the first term on the right is at least as large as the second,
then the left side is at most a constant times
$L|B|^2|B_0|^2/(\eta^3|S|)$. Otherwise the second term gives the bound.
Dividing by $|T'|^2|B|^2$ in either case,
\begin{equation}\label{eq:four-terms}
 \frac{\sum_cr_\times(c)^3}{|T'|^2|B|^2}
 \lesssim\frac{L|B_0|^2}{\eta^3|S||T'|^2}+\frac1{\eta|T'|}
 +\frac{|B_0|}{\eta|S||T'|}+\frac{|B_0||B|}{\eta p|T'|}.
\end{equation}
Put $\kappa_*:=|T'|/(\kappa|T|^2)$. We show that each of the first
three terms is $O(L^5\kappa_*)$. By \eqref{eq:S-size} and then
\eqref{eq:Tprime-size}, the first term is at most a constant times
\[
 \frac{L^2\delta^3|T|}{\kappa\eta^3|T'|^2}
 =L^2\kappa_*\Bigl(\frac{\delta|T|}{\eta|T'|}\Bigr)^3\lesssim L^5\kappa_*.
\]
By \eqref{eq:Tprime-size} and then \eqref{eq:eta-kappa}, the second
term is
\[
 \frac1{\eta|T'|}=\kappa_*\frac{\kappa|T|^2}{\eta|T'|^2}
 \lesssim\kappa_*\frac{L^2\kappa\eta}{\delta^2}\lesssim L^4\kappa_*.
\]
By \eqref{eq:S-size}, $|B_0|\ge\delta|T|$, and \eqref{eq:Tprime-size},
the third term is at most a constant times
\[
 \frac{L\delta^3|T|}{\eta\kappa|B_0||T'|}
 \le\frac{L\delta^2}{\eta\kappa|T'|}
 =\kappa_*\frac{L\delta^2|T|^2}{\eta|T'|^2}
 \lesssim L^3\kappa_*\eta\le L^3\kappa_*.
\]
To bound the last term, we use the trivial estimate for large $B$
and the hypothesis $|T|^3\le p$ for the remaining sets $B$.
If $|B|\ge\kappa|T|^2=|T'|/\kappa_*$, the trivial bound
$\sum_cr_\times(c)^3\le|T'|^3|B|\le\kappa_*|T'|^2|B|^2$ gives the
required estimate. Otherwise $|B|<\kappa|T|^2$, and by
\eqref{eq:B0-size} the last term is at most
\[
 \frac{2|T|^3}{\eta p|T'|}\le\frac2{\eta|T'|},
\]
which is twice the second term. In all cases
$\sum_cr_\times(c)^3\lesssim L^5\kappa_*|T'|^2|B|^2$, so
$\kappa^\times(T')\lesssim L^5\kappa_*$, which is \eqref{eq:claim}.

\smallskip
\noindent\emph{Iteration.}
Put $X_0:=\emptyset$. While $|X_i|<N/2$, apply the claim to
$T_i:=A\setminus X_i$, which has more than $N/2$ elements and satisfies
$|T_i|^3\le p$, to obtain a nonempty $T'_i\subseteq T_i$, and put
$X_{i+1}:=X_i\cup T'_i$. The sets $X_i$ strictly increase, so the
process stops at the first index $i\ge1$ with $|X_i|\ge N/2$. Let
$X:=X_i$. To keep $Y$ large, take the remainder just before the last
step, $Y:=A\setminus X_{i-1}$, so that $|Y|>N/2$. For
$0\le j<i$ we have $Y\subseteq T_j$, since $X_j\subseteq X_{i-1}$.
It follows that $\kappa^+(Y)|Y|^2\le\kappa^+(T_j)|T_j|^2$.
Hence by \eqref{eq:claim}, with $L=\log(4N)$ bounding each
$\log(4|T_j|)$,
\[
 \kappa^+(Y)|Y|^2\kappa^\times(T'_j)
 \le\kappa^+(T_j)|T_j|^2\kappa^\times(T'_j)\lesssim L^5|T'_j|.
\]
The sets $T'_0,\ldots,T'_{i-1}$ are disjoint with union $X$. The
inequality for disjoint unions gives
$\kappa^\times(X)\le\sum_j\kappa^\times(T'_j)$, and
summing the last display over $j$ yields
$\kappa^+(Y)|Y|^2\kappa^\times(X)\lesssim L^5|X|\le L^5N$. Since
$|Y|>N/2$ this proves the proposition.
\end{proof}

\begin{corollary}\label{cor:sumproduct-two}
For every $\varepsilon>0$ and every finite nonempty $A\subset k$ with
$|A|\le p^{1/3}$,
\[
 \max\{|A+A|,|AA|\}\gtrsim_\varepsilon|A|^{49/38-\varepsilon}.
\]
In characteristic zero there is no restriction on $|A|$.
\end{corollary}

\begin{proof}
As in the proof of Corollary~\ref{cor:sumproduct} we may assume
$0\notin A$, and we put $N:=|A|$ and
$\Delta:=\max\{|A+A|,|AA|\}/N\ge1$. Proposition~\ref{prop:decomposition}
gives $X,Y\subseteq A$ with $|X|,|Y|\ge N/2$ and
$\kappa^+(Y)\kappa^\times(X)\lesssim(\log4N)^5/N$. Fix $0<\delta<11/19$
and put $\alpha:=11/19-\delta$. Bloom's theorem \eqref{eq:bloom} in the
additive and multiplicative groups gives
\[
 |Y+Y|\gtrsim_\delta\kappa^+(Y)^{-\alpha}|Y|,\qquad
 |XX|\gtrsim_\delta\kappa^\times(X)^{-\alpha}|X|.
\]
Since $Y+Y\subseteq A+A$ and $XX\subseteq AA$, both sets on the left
have size at most $\Delta N$. Combining this with $|X|,|Y|\ge N/2$ gives
$\kappa^+(Y),\kappa^\times(X)\gtrsim_\delta\Delta^{-1/\alpha}$.
Multiplying these inequalities and applying the decomposition bound gives
\[
 \Delta^{-2/\alpha}\lesssim_\delta\kappa^+(Y)\kappa^\times(X)
 \lesssim\frac{(\log4N)^5}N,
 \qquad\text{so}\qquad
 \Delta\gtrsim_\delta\Bigl(\frac N{(\log4N)^5}\Bigr)^{\alpha/2}.
\]
As $\delta\to0$ the exponent $\alpha/2$ tends to $11/38$. Choosing
$\delta$ small in terms of $\varepsilon$ and absorbing the logarithm
gives $\Delta\gtrsim_\varepsilon N^{11/38-\varepsilon}$, and
$1+11/38=49/38$.
\end{proof}

Over the reals, Bloom \cite[Theorem~4]{Bloom} obtains the larger
exponent $1270/951\approx1.3354$ by combining control with arguments specific to the
real numbers.

\subsection{Projections and Furstenberg sets}\label{sec:projections}

We thank Paige Bright for pointing out the following consequences of
Theorem~\ref{thm:main}. The two corollaries below give prime-field
analogues of the planar projection and Furstenberg estimates of
Orponen and Shmerkin \cite{OS} and Ren and Wang \cite{RW} in the
Euclidean plane.
Here $p$ is any prime. We use the linear projections
\[
 \pi_\theta(x_1,x_2)=x_1+\theta x_2,\qquad \theta\in\F_p.
\]
A fiber $\pi_\theta^{-1}(t)$ is the line $x_1+\theta x_2=t$.
Thus $|\pi_\theta(X)|$ counts the lines in this parallel family that
meet $X$. The next corollary bounds the number of directions for which
this projection is small.

\begin{corollary}\label{cor:projections}
There is an absolute constant $c>0$ such that, for every nonempty
$X\subseteq\F_p^2$ and $1\le s\le c\min\{p,|X|\}$,
\[
 \bigl|\{\theta\in\F_p:|\pi_\theta(X)|<s\}\bigr|
 \lesssim\max\left\{\frac{s^2}{|X|},1\right\}.
\]
\end{corollary}

\begin{proof}
Put $M=|X|$, let $\Theta$ be the set of directions in the statement, and write
$T=|\Theta|$. We may assume $T>0$. For each $\theta\in\Theta$, take the fibers of
$\pi_\theta$ that meet $X$. Let $\LL$ be the set of all these lines,
so $|\LL|<sT$. Each point of $X$ lies on exactly one selected
line for each $\theta$. Hence $I(X,\LL)=MT$, and
Theorem~\ref{thm:main} gives
\[
 MT\lesssim(MsT)^{2/3}+M+sT+MsT/p.
\]
Since $s\le cM$ and $s\le cp$, the last two terms are at most $2cMT$,
and if $c$ is sufficiently small they are absorbed into the left side.
Thus either $T\lesssim1$ or $MT\lesssim(MsT)^{2/3}$, which gives
$T\lesssim s^2/M$.
\end{proof}

The remaining projection $\pi_\infty(x_1,x_2)=x_2$ adds at most one
exceptional direction, so the same bound holds for all directions.

We next give a lower bound for $|X|$ when each of $t$ lines contains
at least $s$ points of $X$.

\begin{corollary}\label{cor:furstenberg}
Let $X\subseteq\F_p^2$ and let $\LL$ consist of $t\ge1$ distinct lines.
If every line in $\LL$ contains at least $s$ points of $X$, where
$s\ge2$ is an integer, then
\[
 |X|\gtrsim\min\{s^{3/2}t^{1/2},\ st,\ sp\}.
\]
\end{corollary}

\begin{proof}
Write $M=|X|$. Since $st\le I(X,\LL)$,
Theorem~\ref{thm:main} implies
\[
 st\le C\bigl((Mt)^{2/3}+M+t+Mt/p\bigr)
\]
for an absolute constant $C$. If $s\ge2C$ then $Ct\le st/2$ can be
absorbed into the left side. At least one
remaining term is comparable to $st$, giving respectively
$M\gtrsim s^{3/2}t^{1/2}$, $M\gtrsim st$, or $M\gtrsim sp$.
If $2\le s<2C$, count pairs of points on the selected lines. Each
selected line contains at least $\binom s2$ pairs of points of $X$, and
a pair of points lies on at most one line, so
\[
 t\binom s2\le\binom M2.
\]
It follows that $M\gtrsim s\sqrt t\gtrsim s^{3/2}\sqrt t$, since $s$
is bounded by an absolute constant. This proves the claim in this case
as well.
\end{proof}

The condition $s\ge2$ is necessary: one point and all the lines through
it give a counterexample when $s=1$. In the notation of Gan \cite{Gan},
Corollary~\ref{cor:furstenberg} proves Conjecture~1.5, the planar
prime-field Furstenberg conjecture. His Theorems~1.8 and~1.11 therefore
give the corresponding sharp exponents in all dimensions, with the
$p^\varepsilon$ losses and the projection thresholds stated there.

\appendix
\section{Degree optimization and the prime-field estimate}\label{app:constants}

We complete the proof of Theorem~\ref{thm:main} from the estimate obtained
in Section~\ref{sec:incidences}. The degree choice gives the coefficient
$3$ and the subcritical range. A separate second-moment count improves
the characteristic term over the prime field. These arguments are kept
here to separate the explicit constants from the polynomial construction.

\subsection{Choosing the degree}\label{app:degree}

If $m=0$ or $n=0$ all the claimed bounds are immediate. Suppose $m,n\ge1$.
For every integer $1\le d<p$, Section~\ref{sec:incidences} gives
\[
 I(P,\LL)\le m+n+\frac{2mn}{d+2}+d^2.
\]
The term $2mn/(d+2)$ decreases in $d$ and $d^2$ increases, and they
balance when $d$ is of order $(mn)^{1/3}$. Put
$D:=(mn)^{1/3}\ge1$ and let
\[
 d=\lfloor D\rfloor\quad\text{if }p=\infty,\qquad
 d=\min\{\lfloor D\rfloor,p-1\}\quad\text{otherwise},
\]
so that $1\le d<p$ and $d\le D$. If $d=\lfloor D\rfloor$ then $d+2>D$,
and otherwise $d=p-1$ and $d+2>p$. Hence, using $mn/D=D^2$,
\[
 d^2\le D^2,\qquad\frac{2mn}{d+2}\le2D^2+\frac{2mn}p,
\]
and therefore
\[
 I(P,\LL)\le3(mn)^{2/3}+m+n+\frac{2mn}p.
\]
This proves the characteristic-zero bound, since $1/p=0$ in that case,
and \eqref{eq:main} in positive characteristic. If $p<\infty$ and
$mn\le p^3$, then $D\le p$ and
$d+2>D$ in both cases. Hence $2mn/(d+2)\le2D^2$, which proves
\eqref{eq:main-subcritical}.

\subsection{The coefficient of the characteristic term over the prime field}\label{app:prime}

It remains to prove \eqref{eq:main-prime} for the original configuration
over $\F_p$. For $mn\le p^3$ it follows from
\eqref{eq:main-subcritical}. For $mn>p^3$ we use the finite-plane
estimate $I(P,\LL)\le mn/p+\sqrt{pmn}$ of Vinh \cite{Vinh}.
For completeness, this estimate follows from a direct count. Let $r(q)$
be the number of lines of $\LL$ through $q\in\F_p^2$. Each line has
$p$ points, and two distinct lines meet in at most one point, so
\[
 \sum_{q\in\F_p^2}r(q)=pn,\qquad
 \sum_{q\in\F_p^2}r(q)^2\le n(n-1)+pn\le n^2+pn.
\]
Expanding the square gives
\[
 \sum_{q\in\F_p^2}\left(r(q)-\frac np\right)^2\le pn.
\]
Cauchy--Schwarz, applied to the sum over $P$, now yields
\[
 I(P,\LL)-\frac{mn}p
 =\sum_{q\in P}\left(r(q)-\frac np\right)
 \le\sqrt{pmn}.
\]
Since $mn>p^3$, we have $\sqrt{pmn}\le(mn)^{2/3}$, and
\eqref{eq:main-prime} follows.

\medskip
\noindent\textsc{Enfield, New Hampshire}

\noindent\textit{Email address}: \href{mailto:mlewko@gmail.com}{\texttt{mlewko@gmail.com}}

\end{document}